\documentclass{gtmon_a}
\pdfoutput=1

\usepackage{pb-diagram}

\usepackage{latexsym}

\proceedingstitle{Proceedings of the Nishida Fest (Kinosaki 2003)}
\conferencestart{28 July 2003}
\conferenceend{8 August 2003}
\conferencename{International Conference in Homotopy Theory}
\conferencelocation{Kinosaki, Japan}

\editor{Matthew Ando}
\givenname{Matthew}
\surname{Ando}

\editor{Norihiko Minami}
\givenname{Norihiko}
\surname{Minami}

\editor{Jack Morava}
\givenname{Jack}
\surname{Morava}

\editor{W Stephen Wilson}
\givenname{W Stephen}
\surname{Wilson}

\title{On the $E^1$--term of the gravity spectral sequence}

\author{Dai Tamaki}
\givenname{Dai}
\surname{Tamaki}
\address{Department of Mathematical Sciences\\
Shinshu University\\\newline
Matsumoto 390-8621\\Japan}
\email{rivulus@math.shinshu-u.ac.jp}
\urladdr{}

\volumenumber{10}
\issuenumber{}
\publicationyear{2007}
\papernumber{20}
\startpage{347}
\endpage{382}

\doi{}
\MR{}
\Zbl{}

\arxivreference{}


\received{31 August 2004}
\revised{10 June 2005}
\accepted{}
\published{18 April 2007}
\publishedonline{18 April 2007}
\proposed{}
\seconded{}
\corresponding{}
\version{}

\AtBeginDocument{\let\bar\wbar\let\tilde\wtilde\let\hat\what}

\makeop{dis}
\makeop{res}
\makeop{OL}
\makeop{MOL}
\makeop{scan}
\makeop{Cotor}

\newcommand{\shot}{\mathop\simeq\limits_S}
\newcommand{\Ima}{\operatorname{Im}}
\newcommand{\FiniteSets}{\operatorname{\mathbf{Finite\ Sets}}}
\newcommand{\Spaces}{\operatorname{\mathbf{Spaces}}}
\newcommand{\whot}{\mathop\simeq\limits_w}
\newcommand{\gAbels}{\operatorname{\mathbf{Graded\ Abelian\ Groups}}}

\newcommand{\quotient}[2]{%
  \left(#1\right)
  \hspace{-4pt}\raisebox{-5pt}{$\bigg/$}\hspace{-2pt}\raisebox{-12pt}{$#2$}}
\newcommand{\rarrow}[1]{\buildrel #1 \over \longrightarrow}
\newcommand{\larrow}[1]{\buildrel #1 \over \longleftarrow}
\newcommand{\Cu}{\mathcal{C}}

\makeop{Ker}

\makeatletter
\def\cnewtheorem#1[#2]#3{\newtheorem{#1}{#3}[section]
\expandafter\let\csname c@#1\endcsname\c@theorem}

\theoremstyle{plain}
\newtheorem{theorem}{Theorem}[section]
\cnewtheorem{lemma}[theorem]{Lemma}
\cnewtheorem{corollary}[theorem]{Corollary}
\cnewtheorem{proposition}[theorem]{Proposition}

\theoremstyle{definition}
\cnewtheorem{definition}[theorem]{Definition}
\cnewtheorem{assumption}[theorem]{Assumption}

\theoremstyle{remark}
\cnewtheorem{remark}[theorem]{Remark}

\makeatother  

\begin{document}

\begin{htmlabstract}
The author constructed a spectral sequence strongly converging to
h<sub>*</sub>(&Omega;<sup>n</sup>&Sigma;<sup>n</sup> X) for any homology
theory in [Topology 33 (1994) 631&ndash;662]. In this note, we prove
that the E<sup>1</sup>-term of the spectral sequence is isomorphic to
the cobar construction, and hence the spectral sequence is isomorphic
to the classical cobar-type Eilenberg--Moore spectral sequence based on
the geometric cobar construction from the E<sup>1</sup>-term. Similar
arguments can be also applied to its variants constructed in [Contemp
Math 293 (2002) 299&ndash;329].
\end{htmlabstract}

\begin{abstract}
The author constructed a spectral sequence strongly converging to
$h_*(\Omega^n\Sigma^n X)$ for any homology theory in \cite{Tamaki94}. In
this note, we prove that the $E^1$--term of the spectral sequence
is isomorphic to the cobar construction, and hence the spectral
sequence is isomorphic to the classical cobar-type Eilenberg--Moore
spectral sequence based on the geometric cobar construction from the
$E^1$--term. Similar arguments can be also applied to its variants
constructed in \cite{Tamaki02}.
\end{abstract}

\maketitle

\section{Introduction}
\label{Introduction}

 In \cite{Tamaki94}, the author introduced a filtration
 $\{F_{-s}\mathcal{C}_n(j)\}_{s\ge 0}$
 on the space of little cubes $\mathcal{C}_n(j)$ with
 \begin{align*}
  \emptyset = F_{-j-1}\mathcal{C}_n(j) \subset \cdots
  F_{-s-1}\mathcal{C}_n(j) \subset F_{-s}\mathcal{C}_n(j) &\subset
  \cdots \\
  &\subset F_{-1}\mathcal{C}_n(j) = F_0\mathcal{C}_n(j) =
  \mathcal{C}_n(j).
\end{align*}
 This is called the gravity filtration. By using the Snaith splitting,
 \[
  \Omega^n\Sigma^nX \shot \bigvee_{j=1}^{\infty}
  \mathcal{C}_n(j)_{+}\wedge_{\Sigma_j} X^{\wedge j}
 \]
 we obtain a stable filtration on $\Omega^n\Sigma^n X$ and hence a spectral
 sequence for computing $h_*(\Omega^n\Sigma^n X)$. This spectral
 sequence is called the gravity spectral sequence.
 The author proved in \cite{Tamaki94} the following:

 \begin{theorem}
  For any homology theory $h_*(-)$, the gravity filtration induces a
  spectral sequence strongly converging to $h_*(\Omega^n\Sigma^n X)$.

  When $h_*$ is multiplicative and $h_*(\Omega^{n-1}\Sigma^n X)$ is
  $h_*(\ast)$--flat, we have
  \[
  E^2 \cong \Cotor^{h_*(\Omega^{n-1}\Sigma^nX)}(h_*(\ast),h_*(\ast)).
  \]
 \end{theorem}

 Later in \cite{Tamaki02}, the author constructed
 spectral sequences for fibrations related to iterated Freudenthal
 suspensions by adopting the gravity filtration.

 The Eilenberg--Moore spectral sequence constructed by the (algebraic or
 geometric) cobar construction
 (see Eilenberg--Moore \cite{Eilenberg-Moore66-1,Eilenberg-Moore66-2}, Rector \cite{Rector70}
 and Dwyer \cite{Dwyer74,Dwyer75}), the so-called cobar-type Eilenberg--Moore
 spectral sequence, also has its $E^2$--term isomorphic to $\Cotor$.
 A natural question is if the gravity spectral sequence is isomorphic to
 the cobar-type Eilenberg--Moore spectral sequence.

 In the case of the cobar-type Eilenberg--Moore spectral sequence for the
 path-loop fibration
 \[
  \Omega X \longrightarrow PX \longrightarrow X,
 \]
 the $E^1$--term is given by the (algebraic) cobar construction, ie
 \[
  E^1_{-s,\ast} \cong \left(\Sigma^{-1}\tilde{h}_*(X)\right)^{\otimes s},
 \]
 if $h_*$ is multiplicative and $h_*(X)$ is $h_*(\ast)$--flat.

The purpose of this note is to prove the following theorem.
\begin{theorem}
 \label{MainTheorem}
 The $E^1$--term of the gravity spectral sequence \cite{Tamaki94} is isomorphic to the
 cobar construction, ie
 \[
 E^1_{-s,\ast} \cong
 \left(\Sigma^{-1}\tilde{h}_*(\Omega^{n-1}\Sigma^nX)\right)^{\otimes s}
 \]
 as chain complexes, if $h_*$ is multiplicative and
 $h_*(\Omega^{n-1}\Sigma^n X)$ is $h_*(\ast)$--flat. Hence the gravity
 spectral sequence is isomorphic to the classical cobar-type
 Eilenberg--Moore spectral sequence as spectral sequences.
\end{theorem}

This result is useful for practical applications of the gravity spectral
sequence and simplifies the arguments in the computation performed in
the last section of \cite{Tamaki02}.

This result also gives us a ``geometric model'' for the cobar
differential. Note that the spectral sequence splits into a direct sum
of small spectral sequences each of which is induced from the filtration
on $\mathcal{C}_n(j)$ for $j = 1, 2, \cdots$. By definition, the $d^1$
differential on $\mathcal{C}_n(j)$ part is given by the following
composition
\begin{multline*}
F_{-s}\mathcal{C}_n(j)/F_{-s-1}\mathcal{C}_n(j) \simeq
F_{-s}\mathcal{C}_n(j)\cup CF_{-s-1}\mathcal{C}_n(j) \\
\longrightarrow
\Sigma F_{-s-1}\mathcal{C}_n(j) \longrightarrow \Sigma
F_{-s-1}\mathcal{C}_n(j)/F_{-s-2}\mathcal{C}_n(j)
\end{multline*}
where the NDR representation of
$(F_{-s}\mathcal{C}_n(j), F_{-s-1}\mathcal{C}_n(j))$,
which has an
explicit description in terms of the centers and radii of little cubes,
gives the first homotopy equivalences.

In particular, as the referee pointed out, in the case of $n=2$, the
$d^1$ differential is given in terms of shuffles for smash products of
$\Sigma X$. Note that the decomposition of the permutation
representation
\[
 \R^j \cong \left\{(b_1,\cdots,b_j)\in\R^j\left|\sum
 b_k=0\right.\right\}\oplus \{(t,\cdots,t)\mid t\in\R\}
\]
allows us to associate a map
\[
 s \co (\Sigma X)^{\wedge j} \longrightarrow (\Sigma X)^{\wedge j}
\]
to an element $s$ of the group ring $\Z[\Sigma_j]$ which induces the
action of $s$ in homology.
Let $s_{i,j} \in \Z[\Sigma_{i+j}]$ be the sum of all
$(i,j)$--shuffles.

With this notation, we see that, when $n=2$, the $d^1$ differential
\[
 d^1_{s,*} \co E^1_{-s,*} \cong (\Sigma^{-1}T(\tilde{h}_{*}(\Sigma
 X)))^{\otimes s} \longrightarrow (\Sigma^{-1}T(\tilde{h}_{*}(\Sigma
 X)))^{\otimes s+1} \cong E^1_{-s-1,*}
\]
is induced by the following map:
\begin{multline}
 \sum_{m=1}^s\sum_{\ell=1}^{j_m-1} 1\wedge \cdots\wedge
 s_{\ell,j_m-\ell}\wedge \cdots\wedge 1 \co  (\Sigma X)^{\wedge j_1}\wedge
 \cdots \wedge (\Sigma X)^{\wedge j_s} \longrightarrow \\
 \bigvee_{(k_1,\cdots,k_{s+1})} (\Sigma X)^{\wedge k_1}\wedge \cdots
 \wedge (\Sigma X)^{\wedge k_{s+1}},
 \label{n=2}
\end{multline}
where the wedge sum in the range is taken over
$(s+1)$--tuples of positive integers $(k_1,\cdots,k_{s+1})$ with
\[
 k_1+\cdots+k_{s+1} = j_1+\cdots+j_s.
\]
This paper is organized as follows: we recall the definition of the
gravity filtration together with the basic properties of the little cubes
operad in \fullref{GravityFiltration}. The $E^1$--term of the gravity
filtration is analyzed in the first half of \fullref{Decompose} and then
it is proved that the $E^1$--term and the first differential of the
gravity spectral sequence coincides with those of the classical
Eilenberg--Moore spectral sequence in the rest of \fullref{Decompose}. In
\fullref{OverB}, we recall Rector's and Larry Smith's constructions in
order to prove that their spectral sequences are isomorphic to the
gravity spectral sequence from the $E^1$--terms in \fullref{Comparison}.

\textbf{Acknowledgements}\qua This paper should be considered as a
complement to \cite{Tamaki94}, which was a part of the author's
PhD\ thesis supervised by Fred Cohen. The author is grateful to him
for asking about the $E^1$--term of the gravity spectral sequence.
The result in this paper would not have come out without
his persuasion. The author would also like to thank the referee for
pointing out that the $d^1$ differential for $\Omega^2\Sigma^2X$ could
be written in terms of shuffles on $(\Sigma X)^{\wedge j}$.

\section{The gravity filtration}
\label{GravityFiltration}

Let us first recall the construction of the gravity spectral sequence in
\cite{Tamaki94}.

\begin{definition}
 A little $n$--cube is an embedding
\[
 c \co  [-1,1]^n \hookrightarrow [-1,1]^n
\]
\[
 c = \ell_1\times\cdots\times\ell_n  \leqno{\hbox{of the form}}
\]
 where each $\ell_i$ is an orientation preserving affine embedding
 \[
 \ell_i \co  [-1,1] \longrightarrow [-1,1].
 \]
 The space of little $n$--cubes (with compact-open topology) is denoted
 by $\mathcal{C}_n(1)$.
\end{definition}

Note that a $1$--cube $c$ is determined by its center $C=c(0)$ and radius
$R=c(1)-c(0)$.

\begin{definition}
 \label{SpaceofCubes}
 The space of $j$ little $n$--cubes $\mathcal{C}_n(j)$ is defined by
 \[
 \mathcal{C}_n(j) = \left.\left\{(c_1,\cdots,c_j)\in \mathcal{C}_n(1)^j
 \right| \Ima c_i\cap \Ima c_k = \emptyset \ (i\neq k)\right\},
 \]
 where $\Ima c_i = c_i((-1,1)^n)$.
 The symmetric group of $j$ letters $\Sigma_j$ acts on
 $\mathcal{C}_n(j)$ by permuting cubes.

 Note that, for any finite set $S$, the space of $n$--cubes indexed in
 $S$, $\mathcal{C}_n(S)$, is defined. Thus $\mathcal{C}_n$ can be
 regarded as a contravariant functor
 \[
  \mathcal{C}_n \co  \FiniteSets \longrightarrow \Spaces
 \]
 from the category of finite sets and injective maps to the category of
 topological spaces.
\end{definition}

The gravity filtration on $\mathcal{C}_n(j)$ is defined as follows.

\begin{definition}
 Let $F_0\mathcal{C}_n(j) = \mathcal{C}_n(j)$. For $s\ge 1$, define
 \begin{eqnarray*}
  (c_1,\cdots,c_j) \in F_{-s}\mathcal{C}_n(j)
   & \Longleftrightarrow & \textrm{We need to decompose the set
   $\{c_1,\cdots,c_j\}$ into} \\
  & & \textrm{at least $s$ disjoint subsets in order to make} \\
  & & \textrm{each group ``stable under gravity''.}
 \end{eqnarray*}
 We say a collection of little $n$--cubes $\{c_{i_1},\cdots, c_{i_k}\}$ is
 stable under
 gravity if the center of the first coordinate of each cube is contained
 in the images of the first coordinate of other cubes.
For example, the two $2$-cubes in the picture below are stable under gravity.

 
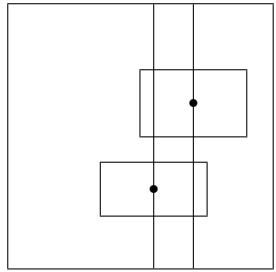
\begin{center}
  \begin{picture}(100,100)(0,00)
   \put(0,0){\framebox(100,100)[bl]{}}
   \put(35,20){\framebox(40,20){}}
   \put(55,30){\circle*{3}}
   \put(50,50){\framebox(40,25){}}
   \put(70,62.5){\circle*{3}}

   \thinlines
   \put(55,0){\line(0,1){100}}
   \put(70,0){\line(0,1){100}}
  \end{picture}
 \end{center}

 In other words, the vertical hyperplane through the center of a cube
 must intersect with the interior of other cubes.
\end{definition}

It is more useful to describe the above filtration in terms of
functions which measure overlaps of the first coordinates of cubes.

\begin{definition}
 Let $b$ be a little $1$--cube with center $C$ and radius $R$. For $x
 \in (-1,1)$, define
 \[
 d(x,b) = \frac{2R-||C+R-x|-|C-R-x||}{2R}.
 \]
 For $c_1,c_2 \in \mathcal{C}_n(1)$, define
 \[
 \dis(c_1,c_2) = \min\{d(c'_1(0),c'_2),d(c'_2(0),c'_1)\},
 \]
 where $c'_i$ is the first coordinate of $c_i$.
\end{definition}

\begin{definition}
 For $S \subset \{1,\cdots,j\}$ and $\mathbf{c}\in \mathcal{C}_n(j)$,
 define
 \[
 \OL(\mathbf{c},S) = \min\{\dis(c_k,c_{\ell})|k,\ell\in S\}.
 \]
 For a partition $P : S_1\amalg\cdots\amalg S_s = \{1,\cdots,j\}$, define
 \[
 \MOL(\mathbf{c},P) = \min_k\{\OL(\mathbf{c},S_k)\}.
 \]
 And for $\mathbf{c} \in F_{-s}\mathcal{C}_n(j)$, define
 \[
 u_s(\mathbf{c}) = \max\{\MOL(\mathbf{c},P)| P: \text{ partition into $s$ subsets}\}.
 \]
\end{definition}

Recall that the configuration space of $j$ points in $\R^n$, $F(\R^n,j)$,
is $\Sigma_j$--equivariantly homotopy equivalent to
$\mathcal{C}_n(j)$. And it is possible and seems more natural to define a
filtration on $F(\R^n,j)$ by the number of distinct first
coordinates. In fact, this filtration on $F(\R^n,j)$ was the origin of
the gravity filtration on $\mathcal{C}_n(j)$. The following fact is an
essential difference between these two filtrations and is proved in
\cite{Tamaki94}.

\begin{lemma}
 \label{NDR}
 The map
 \[
 u_s \co  F_{-s}\mathcal{C}_n(j) \longrightarrow [0,1]
 \]
 is continuous and
 \[
 u_s^{-1}(0) = F_{-s-1}\mathcal{C}_n(j).
 \]
 Furthermore there exists a homotopy
 \[
 h_s \co  F_{-s}\mathcal{C}_n(j)\times I \longrightarrow F_{-s}\mathcal{C}_n(j)
 \]
 with which $(h_s,\tilde{u}_s)$ is an NDR representation for
 $(F_{-s}\mathcal{C}_n(j),F_{-s-1}\mathcal{C}_n(j))$, where
 $\tilde{u}_s = m\circ u_s$ and the function
 \[
  m \co  [0,1] \longrightarrow [0,1]
 \]
 \[
  m(t) = \begin{cases}
      2t & 0\le t\le \frac{1}{2} \\
      1 & \frac{1}{2} \le t \le 1.
     \end{cases}\leqno{\hbox{is defined by}}
 \]
\end{lemma}

The following property of $u_q$ is not used in \cite{Tamaki94}, but
turns out to be very useful for identifying the $E^1$--term.

\begin{lemma}
 \label{VerticallyAligned}
 $u_s(\mathbf{c})=1$ if and only if
 $\mathbf{c}$ consists of $s$ piles of cubes each of which consists of cubes
 whose centers are lined up on a single vertical line (hyperplane).
\end{lemma}

 By using the base point relation, we can form a single space by gluing
 $\mathcal{C}_n(j)\times_{\Sigma_j} X^j$ together.

\begin{definition}
 For a pointed space $X$, define
 \[
 C_n(X) =
 \bigg(\coprod_{j=1}^{\infty}\mathcal{C}_n(j)\times_{\Sigma_j}X^j\bigg)
 \hspace{-4pt}\raisebox{-5pt}{$\bigg/$}\hspace{-2pt}\raisebox{-12pt}{$\sim$}
 \]
 where the relation $\sim$ is given by
 \[
 (c_1,\cdots, c_j;x_1,\cdots,x_j) \sim
 (c_1,\cdots,\hat{c_i},\cdots,c_j;x_1,\cdots,\hat{x_i},\cdots,x_j),
 \]
 if $x_i = \ast$.
\end{definition}

$C_n(X)$ is an approximation to $\Omega^n\Sigma^nX$ up to a weak equivalence.

\begin{theorem}[Approximation Theorem \cite{May72}]
 For a path-connected space $X$ with nondegenerate base point, we have
 the following natural weak equivalence
 \[
 C_n(X) \whot \Omega^n\Sigma^nX.
 \]
\end{theorem}

Unfortunately the gravity filtration is not compatible with the base
point relation in the definition of $C_n(X)$. Fortunately, however, we
can introduce a stable filtration on $C_n(X)$, thanks to the following
famous theorem.

\begin{theorem}[Snaith Splitting \cite{Snaith74}]
 \label{Snaith}
 For a path-connected space $X$ with a nondegenerate base point, we have the
 following natural weak equivalence in the stable homotopy category,
 \begin{equation}
  C_n(X) \shot \bigvee_{j=1}^{\infty}
   \mathcal{C}_n(j)_{+}\wedge_{\Sigma_j}X^{\wedge j}.
   \label{SSforC_n}
 \end{equation}
\end{theorem}

\begin{definition}
 We define
 \[
 F_{-s}C_n(X) = \bigvee_{j=1}^{\infty}
 F_{-s}\mathcal{C}_n(j)_{+}\wedge_{\Sigma_j}X^{\wedge j}.
 \]
\end{definition}

This can be regarded as a filtration on $\Omega^n\Sigma^n X$ in the
stable homotopy category.

We use the following notations:
\begin{eqnarray*}
 D_j^{(n)}(X) & = & \mathcal{C}_n(j)_{+}\wedge_{\Sigma_j} X^{\wedge j}
  \\
 F_{-s}D_j^{(n)}(X) & = & F_{-s}\mathcal{C}_n(j)_{+}\wedge_{\Sigma_j}
  X^{\wedge j}
\end{eqnarray*}
This stable filtration gives a spectral sequence strongly
converging to $h_*(\Omega^n\Sigma^n X)$ for any homology theory
$h_*(-)$, whose $E^1$--term is given by
\begin{eqnarray*}
 E^1_{-s,t} & = & h_{-s+t}(F_{-s}C_n(X),F_{-s-1}C_{n}(X)) \\
 & \cong & \bigoplus_{j\ge 1} h_{-s+t}\left(F_{-s}D_j^{(n)}(X),
                       F_{-s-1}D_j^{(n)}(X)\right) \\
 & \cong & \bigoplus_{j\ge 1} \tilde{h}_{-s+t}\left(F_{-s}D_j^{(n)}(X)/
                       F_{-s-1}D_j^{(n)}(X)\right)
\end{eqnarray*}
since $\smash{\big(F_{-s}D_j^{(n)}(X),F_{-s-1}D_j^{(n)}(X)\big)}$ is an NDR
pair if the base point of $X$ is nondegenerate, thanks to \fullref{NDR}.

\section{Decomposition of cubes}
\label{Decompose}

Let $h_*(-)$ be a homology theory satisfying the strong form of
K\"{u}nneth isomorphism for $\Omega^{n-1}\Sigma^nX$. Then the $E^1$--term
of the classical Eilenberg--Moore spectral sequence for
the path-loop fibration on $\Omega^{n-1}\Sigma^nX$ is isomorphic to
\begin{eqnarray*}
 E^1_{-s,t} & \cong & \tilde{h}_{t}\left(\left(\Omega^{n-1}\Sigma^n
                      X\right)^{\wedge s}\right) \\ 
 & \cong & \tilde{h}_{t}\left(C_{n-1}(\Sigma X)^{\wedge s}\right). 
\end{eqnarray*}
Under the Snaith splitting, we have
\[
 \tilde{h}_{t}\left(C_{n-1}(\Sigma X)^{\wedge s}\right) \cong \bigoplus_{\sum
 i_k=j}\tilde{h}_t\left(D_{i_1}^{(n-1)}(\Sigma X)\wedge\cdots\wedge
 D_{i_s}^{(n-1)}(\Sigma X)\right).
\]
Thus, with the notations in \fullref{GravityFiltration}, all we want to do
is to find a natural homotopy equivalence
\[
  \Sigma^sF_{-s}D_j^{(n)}(X)/F_{-s-1}D_j^{(n)}(X) \simeq
  \bigvee_{\sum i_k=j} D_{i_1}^{(n-1)}(\Sigma X)\wedge\cdots\wedge
  D_{i_s}^{(n-1)}(\Sigma X)
\]
 or a $\Sigma_j$--equivariant homotopy equivalence
 \[
  \Sigma^sF_{-s}\mathcal{C}_n(j)/F_{-s-1}\mathcal{C}_n(j) \simeq 
  \bigvee_{\substack{\amalg S_k=\\\{1,\cdots,j\}}} \!\!\left(\mathcal{C}_{n-1}(S_1)_{+}\wedge
  S^{|S_1|}\right) \wedge\cdots\wedge
  \left(\mathcal{C}_{n-1}(S_s)_{+}\wedge S^{|S_s|}\right)
 \]
where $\mathcal{C}_n(S)$ is the space of little $n$--cubes indexed by the
set $S$ and the wedge sum on the right hand side runs over all
partitions of $\{1,\cdots,j\}$ into nonempty $k$ subsets.

However this is not easy. To understand the difficulty, let us
try to define a map
\[
 \Sigma^2F_{-2}\mathcal{C}_n(3)/F_{-3}\mathcal{C}_n(3) \longrightarrow 
 \bigvee_{\substack{S_1\amalg S_2 =\\\{1,2,3\}}}\left(\mathcal{C}_{n-1}(S_1)_{+}\wedge
 S^{|S_1|}\right) \wedge
 \left(\mathcal{C}_{n-1}(S_2)_{+}\wedge S^{|S_2|}\right).
\]

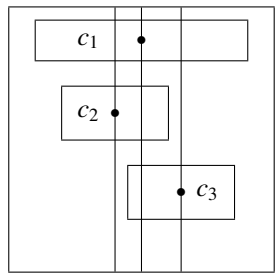
\begin{center}
   \begin{picture}(100,100)(0,0)
   \small
   \put(0,0){\framebox(100,100)[bl]{}}

   \put(10,80){\framebox(80,15){}}
   \put(50,87.5){\circle*{3}}
   \put(50,0){\line(0,1){100}}
   \put(30,87.5){\makebox(0,0){$c_1$}}

   \put(20,50){\framebox(40,20){}}
   \put(40,60){\circle*{3}}
   \put(40,0){\line(0,1){100}}
   \put(30,60){\makebox(0,0){$c_2$}}

   \put(45,20){\framebox(40,20){}}
   \put(65,30){\circle*{3}}
   \put(65,0){\line(0,1){100}}
   \put(75,30){\makebox(0,0){$c_3$}}
  \end{picture}

\end{center}

Forget about the suspension coordinates. Consider the cubes in the
above picture.
This element belongs to $F_{-2}\mathcal{C}_n(3) -
F_{-3}\mathcal{C}_n(3)$ and there are two ways to decompose it into
two collections of cubes stable under gravity, ie
\begin{eqnarray*}
 \{1,2,3\} & = & \{1,2\}\amalg \{3\} \\
 \{1,2,3\} & = & \{1,3\}\amalg \{2\}.
\end{eqnarray*}
Thus a canonical map we obtain is
\[
 \Sigma^2F_{-2}\mathcal{C}_n(3)/F_{-3}\mathcal{C}_n(3) \longrightarrow 
 \prod_{\substack{S_1\amalg S_2 =\\\{1,2,3\}}} \left(\mathcal{C}_{n-1}(S_1)_{+}\wedge
 S^{|S_1|}\right) \wedge
 \left(\mathcal{C}_{n-1}(S_2)_{+}\wedge S^{|S_2|}\right), 
\]
not into $\bigvee$.

More generally, taking all possible decompositions would give the
following map
\[\Sigma^sF_{-s}\mathcal{C}_n(j)
  \hspace{-6pt}\raisebox{-5pt}{$\bigg/$}\hspace{-8pt}\raisebox{-12pt}{$F_{-s-1}\mathcal{C}_n(j)$}
\!\!\longrightarrow\!\!
 \prod_{\substack{\amalg S_k=\\\{1,\cdots,j\}}}\!\! \left(\mathcal{C}_{n-1}(S_1)_{+}\wedge
 S^{|S_1|}\right) \wedge\cdots\wedge
 \left(\mathcal{C}_{n-1}(S_s)_{+}\wedge S^{|S_s|}\right). 
\]
We need to compress the image of this map into
\[
\bigvee_{\amalg S_k=\{1,\cdots,j\}} \left(\mathcal{C}_{n-1}(S_1)_{+}\wedge
S^{|S_1|}\right) \wedge\cdots\wedge
\left(\mathcal{C}_{n-1}(S_s)_{+}\wedge S^{|S_s|}\right).
\]
To this end, our idea is to deform
$\Sigma^sF_{-s}\mathcal{C}_n(j)/F_{-s-1}\mathcal{C}_n(j)$ into a smaller
space of ``decomposable cubes''.
 
G\,Dunn introduced the notion of decomposable cubes in \cite{Dunn88} 
and proved a decomposition of the little $n$--cubes operad
\[
 \mathcal{C}_n \simeq
 \underbrace{\mathcal{C}_1\otimes\cdots\otimes\mathcal{C}_1}_n.
\]
In our case, we need the notion of horizontally decomposable cubes.

\begin{definition}
 Let $\mathcal{D}_n^s(j)$ be the subset of $\mathcal{C}_n(j)$ consisting
 of cubes which are horizontally decomposable into $s$ collections (as in the picture below).
\end{definition}
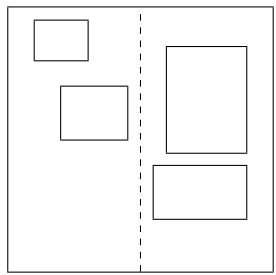
\begin{center}
 \begin{picture}(100,100)(0,0)
  \put(0,0){\framebox(100,100)[bl]{}}
  
  \put(10,80){\framebox(20,15){}}
  \put(20,50){\framebox(25,20){}}

  \multiput(50,0)(0,5){20}{\line(0,1){2.5}}

  \put(55,20){\framebox(35,20){}}
  \put(60,45){\framebox(30,40){}}
 \end{picture}

\end{center}

More precisely, let
\[
  i_1 \co  \mathcal{C}_1(j) \hookrightarrow \mathcal{C}_n(j)
\]
be the inclusion of the first coordinate given by the multiplication
of the identity \mbox{$(n-1)$--cube}.
\[
  i_1(c_1,\cdots,c_j) = (c_1\times 1_{I^{n-1}},\cdots, c_j\times 1_{I^{n-1}}).
\leqno{\hbox{Namely}}
\]
Then $\mathcal{D}_n^s(j)$ is the image of the
following restriction of the operad structure map
\[
  \gamma \co  i_1(\mathcal{C}_1(s))\times \left(\coprod_{\sum i_k = j}
  \mathcal{C}_n(i_1)\times
  \cdots \times\mathcal{C}_n(i_s)\right) \longrightarrow \mathcal{C}_n(j).
\]
We have the following diagram:
\[
 \begin{diagram}
  \node{\mathcal{D}_n^s(j)} \arrow{e,t}{\subset}
  \node{F_{-s}\mathcal{C}_n(j)} \\
  \node{\mathcal{D}_n^{s+1}(j)} \arrow{n,r}{\cup} \arrow{e,t}{\subset}
  \node{F_{-s-1}\mathcal{C}_n(j)} \arrow{n,r}{\cup} 
 \end{diagram}
\]
We want to show that the inclusion gives a homotopy equivalence
\[
 \mathcal{D}_n^s(j)/\mathcal{D}_n^{s+1}(j)) \simeq
 F_{-s}\mathcal{C}_n(j)/F_{-s-1}\mathcal{C}_n(j).
\]
Note that if $\mathbf{c} = (c_1,\cdots,c_j) \in F_{-s}\mathcal{C}_n(j)$,
then there are at least $s$ cubes $c_{i_1}, \cdots, c_{i_s}$ whose
centers of the first coordinates are distinct.

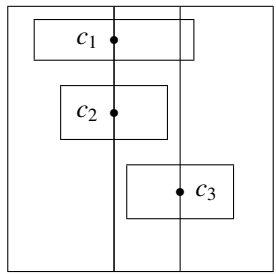
\begin{center}
  \begin{picture}(100,100)(0,0)
  \small
   \put(0,0){\framebox(100,100)[bl]{}}

   \put(10,80){\framebox(60,15){}}
   \put(40,87.5){\circle*{3}}
   \put(40,0){\line(0,1){100}}
   \put(30,87.5){\makebox(0,0){$c_1$}}

   \put(20,50){\framebox(40,20){}}
   \put(40,60){\circle*{3}}
   \put(40,0){\line(0,1){100}}
   \put(30,60){\makebox(0,0){$c_2$}}

   \put(45,20){\framebox(40,20){}}
   \put(65,30){\circle*{3}}
   \put(65,0){\line(0,1){100}}
   \put(75,30){\makebox(0,0){$c_3$}}
  \end{picture}
\end{center}

Thus by shrinking the radii of the first coordinates of cubes, we can
deform $F_{-s}\mathcal{C}_n(j)$ into $\mathcal{D}_n^s(j)$.
The cubes in the above picture are in $\smash{F_{-2}\mathcal{C}_2
(3)}$ but not in $\smash{\mathcal{D}_2^2(3)}$. A horizontal shrinking deforms
the cubes into $\smash{\mathcal{D}_2^2(3)}$.

\begin{definition}
 Let
 \[
 H \co  \mathcal{C}_n(j)\times [0,1) \longrightarrow \mathcal{C}_n(j)
 \]
 be the homotopy which shrinks the radius of the first coordinate of each
 cube linearly without moving the center.
 \begin{gather*}
 \sigma_s \co  F_{-s}\mathcal{C}_n(j) \longrightarrow [0,1) \tag*{\hbox{Define}}
 \\
 \sigma_s(\mathbf{c}) = \inf\{t| H(\mathbf{c},t) \in \mathcal{D}_n^s(j)\}.\tag*{\hbox{by}}
\end{gather*}
\end{definition}

Then obviously $\sigma_s$ is continuous and gives the minimal amount of
the radii of the first coordinates we need to shrink for those cubes in
$F_{-s}\mathcal{C}_n(j)$ in order
to compress them into $\mathcal{D}_n^s(j)$.
For most cubes $\mathbf{c} \in F_{-s}\mathcal{C}_n(j)$,
$H(\mathbf{c},\sigma_{q+1}(\mathbf{c}))$ is defined.
However, if $\mathbf{c}$ consists of $s$ piles of cubes each of which
consists of cubes whose centers are lined up in a single vertical line
(hyperplane), $H(\mathbf{c},\sigma_{q+1}(\mathbf{c}))$ squashes the cubes flat
vertically. For those cubes we need to use $H(\mathbf{c},\sigma_s(\mathbf{c}))$.
Namely, the amount of shrinking varies for different configurations of
cubes.

Fortunately, we can distinguish those vertically aligned cubes by
using the function
\[
 u_s \co  F_{-s}\mathcal{C}_n(j) \longrightarrow [0,1],
\]
thanks to \fullref{VerticallyAligned}.
Now the following gives us a homotopy we want
\[
 G(\mathbf{c},t) = H(\mathbf{c}, t(u_s(\mathbf{c})\sigma_s(\mathbf{c})
 +(1-u_s(\mathbf{c}))\sigma_{s+1}(\mathbf{c}))).
\]
\[
 G(\mathbf{c},1) \in D^s_n(j) \leqno{\hbox{Note that}}
\]
since $\sigma_s(\mathbf{c})\le \sigma_{s+1}(\mathbf{c})$.
Thus we have a $\Sigma_j$--equivariant homotopy equivalence
\[
 \mathcal{D}_n^q(j)/\mathcal{D}_n^{q+1}(j) \simeq
 F_{-q}\mathcal{C}_n(j)/F_{-q-1}\mathcal{C}_n(j).
\]
Horizontally decomposable cubes decompose. Thus it is enough to prove
the following homotopy equivalence:
\[
 \left(\mathcal{D}_n^1(j)/\mathcal{D}_n^{2}(j)\wedge_{\Sigma_j}
 X^{\wedge j}\right)\wedge S^1 \simeq D^{(n-1)}_{j}(\Sigma X).
\]
Note that $\mathcal{D}_n^1(j)$ contains cubes that are not necessary for
analyzing the filtration quotients.
Namely, we don't need those cubes with $\Ima c'_1\cap \cdots \cap\Ima
c'_j = \emptyset$, where $c'_i$ is the first coordinate of the cube
$c_i$, as with the four cubes in the picture below.

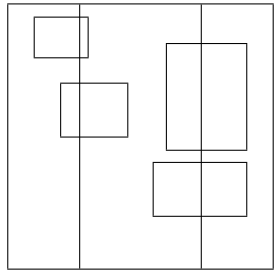
\begin{center}
   \begin{picture}(100,100)(0,0)
   \put(0,0){\framebox(100,100)[bl]{}}

   \put(10,80){\framebox(20,15){}}
   \put(20,50){\framebox(25,20){}}
   \put(27,0){\line(0,1){100}}


   \put(55,20){\framebox(35,20){}}
   \put(60,45){\framebox(30,40){}}
   \put(73,0){\line(0,1){100}}
  \end{picture}
\end{center}
In order to be more efficient, let us introduce yet another filtration.

\begin{definition}
 Define $G_{-s}\mathcal{C}_n(j)$ to be the subset of $\mathcal{C}_n(j)$
 consisting of
 cubes $(c_1,\cdots, c_j)$ which cannot be decomposed into $(s-1)$
 collections of cubes each of which can be skewered by a vertical line
 (hyperplane) intersecting with each interior.
\end{definition}

 Then we have the homotopy equivalence
 \[
 G_{-1}\mathcal{C}_n(j)/G_{-2}\mathcal{C}_n(j) \simeq
 \mathcal{D}_n^1(j)/\mathcal{D}_n^2(j),
 \]
 and the scanning map
 \[
 \scan_1 \co 
 \left(G_{-1}\mathcal{C}_n(j)/G_{-2}\mathcal{C}_n(j)\wedge_{\Sigma_j}
 X^{\wedge
 j}\right)\wedge \left(\Delta^1/\partial
 \Delta^1\right) \longrightarrow D^{(n-1)}_{j}(\Sigma X)
 \]
given by taking the intersection with the vertical hyperplane with the
first coordinate $t \in \Delta^1$ is surjective.
However, it is not easy to find a homotopy inverse to this map. We use
the following auxiliary space instead.

\begin{definition}
 Let $\mathcal{C}^{\varepsilon}_n(j)$ be the subset of
 $\mathcal{C}_n(j)$ consisting of cubes whose first coordinates have
 radius $\varepsilon$.  
\end{definition}

For those cubes in $G_{-1}\mathcal{C}_n(j)-G_{-2}\mathcal{C}_n(j)$, we
can deform the radii in the horizontal direction freely and we have a
homotopy equivalence
\[
G_{-1}\mathcal{C}_n(j)/G_{-2}\mathcal{C}_n(j) \simeq
G_{-1}\mathcal{C}_n^{\varepsilon}(j)/G_{-2}\mathcal{C}_n^{\varepsilon}(j)
\]
for $\varepsilon$ small enough.
 
Since the cubes in $\mathcal{C}_n^{\varepsilon}(j)$ are determined by their
centers, we have the following homeomorphism
\[
G_{-1}\mathcal{C}_n^{\varepsilon}(j)/G_{-2}\mathcal{C}_n^{\varepsilon}(j)
\cong P_j^{\varepsilon}/ dP_j^{\varepsilon}\wedge \mathcal{C}_{n-1}(j)_{+},
\]
where $P_j^{\varepsilon}$ is the convex polytope in $[-1,1]^j$ given by
\[
P_j^{\varepsilon} = \left.\left\{(b_1,\cdots,b_j)\in [-1,1]^j
\right\vert |b_i-b_k| \le 2\varepsilon \text{ for any $i,k$}\right\}
\]
and $dP_j^{\varepsilon}$ is given by
\[
dP_j^{\varepsilon} = \left.\left\{(b_1,\cdots,b_j)\in [-1,1]^j
\right\vert |b_i-b_k| = 2\varepsilon \text{ for some $i,k$}\right\}.
\]
By projecting onto the hyperplane
\[
V = \left\{(b_1,\cdots,b_j)\in [-1,1]^j\left| \sum b_k = 0\right\}\right.
\]
we obtain a homotopy equivalence
\[
P_j^{\varepsilon}/dP_j^{\varepsilon} \simeq P_j^{\varepsilon}\cap
V/dP_j^{\varepsilon}\cap V.
\]
The picture below illustrates the case $j=2$.

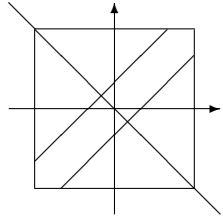
\begin{center}
   \begin{picture}(80,80)(0,0)
   \put(0,40){\vector(1,0){80}}
   \put(40,0){\vector(0,1){80}}

   \put(10,10){\framebox(60,60){}}

   \put(10,20){\line(1,1){50}}
   \put(20,10){\line(1,1){50}}

   \put(0,80){\line(1,-1){80}}
  \end{picture}
\end{center}

 $P_j^{\varepsilon}\cap V$ is a $(j-1)$--dimensional convex polytope
 (dual of permutohedron) and
 $dP_j^{\varepsilon}\cap V$ is its boundary. The decomposition of the
 permutation representation
 \[
 \mathbb{R}^j \cong
 \left\{(b_1,\cdots, b_j) \in
 \mathbb{R}^j\left| \sum b_k =0\right\}\right. \oplus
 \{(t,\cdots,t) \mid t\in \mathbb{R}\}
 \]
 gives a $\Sigma_j$--equivariant homotopy equivalence
 \[
 (P_j^{\varepsilon}\cap V/dP_j^{\varepsilon}\cap V)\wedge
 (\mathbb{R}\cup\{\infty\}) \simeq (S^1)^{\wedge j}
 \]
And we obtain a homotopy equivalence
\begin{eqnarray*}
 \left(G_{-1}\mathcal{C}_n(j)/G_{-2}\mathcal{C}_n(j)\wedge_{\Sigma_j}
  X^{\wedge j}\right)\wedge S^1 & \simeq & (P_j^{\varepsilon}/
 dP_j^{\varepsilon}\wedge \mathcal{C}_{n-1}(j)_{+}\wedge_{\Sigma_j}
 X^{\wedge j}) \wedge S^1 \\
 & \simeq & \mathcal{C}_{n-1}(j)_{+}\wedge_{\Sigma_j} (S^1\wedge X)^{\wedge j} \\
 & = & D^{(n-1)}_{j}(\Sigma X).
\end{eqnarray*}
This completes the proof of the identification of the $E^1$--term of the
gravity filtration with the desired tensor algebra.

Let us consider $d^1$ next. What we have proved so far is the following
fact.
\begin{eqnarray*}
 & & \Sigma^s F_{-s}\mathcal{C}_n(j)_+\wedge_{\Sigma_j} X^{\wedge j} /
  F_{-s-1}\mathcal{C}_n(j)_+\wedge_{\Sigma_j} X^{\wedge j} \\
 & = &  \Sigma^s
  F_{-s}\mathcal{C}_n(j)/F_{-s-1}\mathcal{C}_n(j)\wedge_{\Sigma_j}
  X^{\wedge j} \\
 & \simeq & \quotient{\bigvee_{S_1\amalg \cdots \amalg S_s = \{1,\cdots,
  j\}} \mathcal{C}_{n-1}(S_1)_{+}\wedge S^{S_1} \wedge \cdots
  \wedge \mathcal{C}_{n-1}(S_s)_{+}\wedge S^{S_s} \wedge X^{\wedge j}}{\Sigma_j} \\
 & \simeq & \quotient{\bigvee_{S_1\amalg \cdots \amalg S_s = \{1,\cdots,
  j\}} \mathcal{C}_{n-1}(S_1)_{+}\wedge (\Sigma X)^{\wedge S_1}\wedge \cdots
  \wedge \mathcal{C}_{n-1}(S_s)_{+}\wedge (\Sigma X)^{\wedge S_s}}{\Sigma_j}
\end{eqnarray*}
where the wedge in the first homotopy equivalence runs over all
decomposition of the set $\{1,\cdots, j\}$ into a disjoint of nonempty
$s$ subsets
\[
 S_1\amalg \cdots \amalg S_s = \{1,\cdots, j\},
\]
and for $S \subset \{1,\cdots, j\}$, we abuse the notation to denote
$Y^{\wedge |S|}$ together with the action of the symmetric group by
$Y^{\wedge S}$.

Under the identification by the action of $\Sigma_j$, we obtain the
wedge over all $j_1+\cdots+j_s = j$ and
\begin{multline*}
   \Sigma^sF_{-s}\mathcal{C}_n(j) /
   F_{-s-1}\mathcal{C}_n(j)\wedge_{\Sigma_j} X^{\wedge j} \\
   \simeq  \bigvee_{j_1+\cdots+j_s=j}
  \mathcal{C}_{n-1}(j_1)_{+}\wedge_{\Sigma_{j_1}}
  (\Sigma X)^{\wedge j_1}\wedge \cdots
  \wedge \mathcal{C}_{n-1}(j_s)_{+}\wedge_{\Sigma_{j_s}} (\Sigma
  X)^{\wedge j_s}.
\end{multline*}
However in order to compute $d^1$, we should compute the map 
\[
 \Sigma^sF_{-s}\mathcal{C}_n(j)/F_{-s-1}\mathcal{C}_n(j)\wedge X^{\wedge
 j} \longrightarrow
 \Sigma^{s+1}F_{-s-1}\mathcal{C}_n(j)/F_{-s-2}\mathcal{C}_n(j)\wedge
 X^{\wedge j}
\]
before we take the quotient by the action of $\Sigma_j$. This
``connecting homomorphism'' is given by the composition
\begin{multline}
F_{-s}\mathcal{C}_n(j)/F_{-s-1}\mathcal{C}_n(j) \simeq
F_{-s}\mathcal{C}_n(j)\cup CF_{-s-1}\mathcal{C}_n(j) \\
\longrightarrow
\Sigma F_{-s-1}\mathcal{C}_n(j) \longrightarrow \Sigma
F_{-s-1}\mathcal{C}_n(j)/F_{-s-2}\mathcal{C}_n(j). 
\label{ConnectingMap}
\end{multline}
The first homotopy equivalence in the above sequence of maps is obtained
from an NDR representation for the pair
$(F_{-s}\mathcal{C}_n(j), F_{-s-1}\mathcal{C}_n(j))$.
More precisely, for an NDR pair $(X,A)$ with NDR representation $(h,u)$
satisfying
\[
h(x,s) \in A \quad \text{for}\ u(x)<s,
\]
\[
 \tilde{h} \co  X/A \rarrow{\simeq} X\cup CA \leqno{\hbox{the map}}
\]
\[
 \tilde{h}([x]) = \begin{cases}
           h(x,1) & u(x)= 1 \\
           (h(x,1),u(x)) & u(x) < 1
          \end{cases} \leqno{\hbox{defined by}}
\]
is a homotopy equivalence (see Str{\o}m \cite{Strom66,Strom68}). It is
straightforward to check
that the NDR representation $(h_s,\tilde{u}_s)$ for $(F_{-s}\mathcal{C}_n(j),
F_{-s-1}\mathcal{C}_n(j))$ satisfies the above Str{\o}m condition.

Recall that we have proved that $F_{-s}\mathcal{C}_n(j)$ can be replaced
with the space of little cubes whose first coordinates have a fixed
small radius $\varepsilon$,
$G_{-s}\mathcal{C}_n^{\varepsilon}(j)$.
In the case of $s=1$, the map
\begin{equation}
 G_{-1}\mathcal{C}_n^{\varepsilon}(j)/G_{-2}\mathcal{C}_n^{\varepsilon}(j)
  \longrightarrow
 \Sigma
 G_{-2}\mathcal{C}_n^{\varepsilon}(j)/G_{-3}\mathcal{C}_n^{\varepsilon}(j)
 \label{d^1}
\end{equation}
is given by shrinking the radii of the first coordinate. And the
identification
\[
 \Sigma^2
 G_{-2}\mathcal{C}_n^{\varepsilon}(j)/G_{-3}\mathcal{C}_n^{\varepsilon}(j)
 \simeq
\bigvee_{S_1\amalg S_2=\{1,\cdots,j\}} \mathcal{C}_{n-1}(S_1)_{+}\wedge
 S^{S^1}
\wedge \mathcal{C}_{n-1}(S_2)_{+}\wedge S^{S_2}
\]
is given by measuring the distance of the centers of the first
coordinates. Namely the component to which the image
of an element $(c_1,\cdots,c_j) \in
G_{-1}\mathcal{C}_n^{\varepsilon}(j)$ belongs
under the above identification is determined by measuring the difference
of the centers of the first coordinates. Under the identification
\[
 \Sigma
 G_{-1}\mathcal{C}_n^{\varepsilon}(j)/G_{-2}\mathcal{C}_n^{\varepsilon}(j)
 \simeq
 \mathcal{C}_{n-1}(j)_{+}\wedge S^j
\]
the map \eqref{d^1} can be identified with
\begin{multline*}
 \mathcal{C}_{n-1}(j)_{+}\wedge S^j \longrightarrow \mathcal{C}_{n-1}(j)_{+}\wedge
 (S^1\vee S^1)^{\wedge j} \\
 = \bigvee_{S_1\amalg S_2=\{1,\cdots,j\}}
 \mathcal{C}_{n-1}(S_1)_{+}\wedge S^{S^1} \wedge
 \mathcal{C}_{n-1}(S_2)_{+}\wedge S^{S_2}.
\end{multline*}
The suspension coordinate in $\Sigma
G_{-1}\mathcal{C}_n^{\varepsilon}(j)/G_{-2}\mathcal{C}_n^{\varepsilon}(j)$
determines the position in the first coordinate with which
$(c_1,\cdots,c_j)$ is cut into two collections. Thus $d^1$ is given by
taking all possible decompositions $\{1,\cdots,j\} = S_1\amalg S_2$ and
summing it up, before we divide by the action of $\Sigma_j$.

Therefore we see that
\eqref{ConnectingMap} is given by taking all possible decomposition of
indexing sets
\[
 \{1,\cdots,j\} = S_1\amalg \cdots \amalg S_s
\]
under the horizontal decomposition above.

On the other hand, it is well-known that, under the Snaith splitting,
the coproduct on $C_{n-1}(\Sigma X)$ is given by
\begin{eqnarray*}
D^{(n-1)}_j(\Sigma X) & \longrightarrow & D^{(n-1)}_j(\Sigma X\vee
 \Sigma X) \\
& = & \bigvee_{j_1+j_2=j}\mathcal{C}_{n-1}(j)_{+}
 \wedge_{\Sigma_{j_1}\times\Sigma_{j_2}} (\Sigma
 X)^{\wedge j_1}\wedge (\Sigma X)^{\wedge j_2} \\
& \longrightarrow & \bigvee_{j_1+j_2=j}D^{(n-1)}_{j_1}(\Sigma X)\wedge
 D^{(n-1)}_{j_2}(\Sigma X).
\end{eqnarray*}
Thus the map induced by the composition \eqref{ConnectingMap} in
homology coincides with the cobar differential.

Consider the case $n=2$. Recall that $\mathcal{C}_1(j) \simeq \Sigma_j$,
$\Sigma_j$--equivariantly, and we have
\[
 D_j^{(1)}(\Sigma X) = \mathcal{C}_1(j)_{+}\wedge_{\Sigma_j} (\Sigma
 X)^{\wedge j} \simeq (\Sigma X)^{\wedge j}.
\]
Thus the $d^1$ differential in the case of $\Omega^2\Sigma^2 X$ is given
by the map described in \eqref{n=2}.

\section{Constructions by Rector and Smith}
\label{OverB}

In the previous section, we have seen that the $E^1$--term of the gravity
spectral sequence is isomorphic to the $E^1$--term of the classical
cobar-type Eilenberg--Moore spectral sequence.
In order to finish the proof of \fullref{MainTheorem}, we need to
compare the $E^r$--terms for $r\ge 2$.

A couple of ways are known to construct a spectral sequence for a
 diagram
\[
 \begin{diagram}
  \node{X\times_BY} \arrow{e} \arrow{s} \node{X} \arrow{s} \\
  \node{Y} \arrow{e} \node{B}
 \end{diagram}
\]
\[
 E^2 \cong \Cotor^{h_*(B)}(h_*(X),h_*(Y))\mskip90mu \leqno{\hbox{with $E^2$--term}}
\]
for reasonably good homology theory $h_*(-)$. Our gravity spectral
sequence is one of them. Rector's construction \cite{Rector70} is a
generalization of the classical Eilenberg--Moore spectral sequence
\cite{Eilenberg-Moore66-1,Eilenberg-Moore66-2}. A construction due to
Larry Smith \cite{L.SmithEMSS,L.Smith70} gives us a general framework
for this kind of construction. We prove the remaining part of \fullref{MainTheorem} by comparing the gravity spectral sequence with
Rector's construction of the Eilenberg--Moore
spectral sequence with an aid of Larry Smith's construction.

Let us briefly recall Rector's construction in \cite{Rector70}. Given a
diagram
\[
 Y \rarrow{f} B \larrow{p} X
\]
we can form a cosimplicial space
\[
 \Omega^0(Y,B,X)
 \begin{array}{c}
  \rarrow{\delta^0} \\
  \rarrow{\delta^1} \\
  \larrow{\sigma^1}
 \end{array}
 \Omega^1(Y,B,X) 
 \begin{array}{c}
  \rarrow{\delta^0} \\
  \rarrow{\delta^1} \\
  \rarrow{\delta^2} \\
  \larrow{\sigma^1} \\
  \larrow{\sigma^2}
 \end{array}
 \Omega^2(Y,B,X)
 \cdots
\]
\[
\Omega^n(Y,B,X) = Y\times B^n\times X \leqno{\hbox{where}}
\]
and the maps are defined by
\begin{eqnarray*}
 \delta^i(y,b_1,\cdots,b_n,x) & = & \left\{
 \begin{array}{ll}
  (y,f(y),b_1,\cdots,b_n,x) & \mbox{if $i=0$} \\
  (y,b_1,\cdots,b_i,b_i,\cdots,b_n,x) & \mbox{if $1\le i\le n$} \\
  (y,b_1,\cdots,b_n,p(x),x) & \mbox{if $i=n$}
 \end{array}\right. \\
 \sigma^i(y,b_1,\cdots,b_n,x) & = &
  (y,b_1,\cdots,b_{i-1},b_{i+1},\cdots,b_n,x).
\end{eqnarray*}
Rector defined a sequence of pointed cofibrations
\begin{eqnarray}
 \Omega_0 & \rarrow{\varphi_{-1}} & \overline{\Omega}_{-1}
  \longrightarrow \Omega_{-1} \nonumber \\
 \Omega_{-1} & \rarrow{\varphi_{-2}} & \overline{\Omega}_{-2}
  \longrightarrow \Omega_{-2} \nonumber \\
 & \vdots & \label{cofibrations} \\
 \Omega_{-n+1} & \rarrow{\varphi_{-n}} & \overline{\Omega}_{-n}
  \longrightarrow \Omega_{-n} \nonumber \\
 & \vdots & \nonumber
\end{eqnarray}
directly from the cosimplicial cobar construction $\Omega^*(Y,B,X)$ as
follows.
\[
 \overline{\Omega}_{-n} = \Omega^n(Y,B,X)/\Ima\delta^1 \cup \cdots \cup
 \Ima\delta^n.\mskip80mu \leqno{\hbox{First define}}
\]
The map $\delta^0$ induces a well-defined map
\[
 \psi_{-n} \co  \overline{\Omega}_{-n+1} \longrightarrow \overline{\Omega}_{-n}.
\]
It is easy to check that $\psi_{-n-1}\psi_{-n} = \ast$.
$\Omega_{-n}$'s are defined inductively on $n$. Define
\[
\Omega_0 = \Omega^0(Y,B,X) = Y\times X.
\]
Note that $\overline{\Omega}_0 = \Omega_0/\emptyset$. Let $\varphi_{-1}$ be
the composition
\[
\Omega_0 \longrightarrow \Omega_0/\emptyset = \overline{\Omega}_0
\rarrow{\psi_{-1}} \overline{\Omega}_{-1}.
\]
Let $\Omega_{-1}$ be the (reduced) mapping cone of $\varphi_{-1}$.
Since $\psi_{-2}\psi_{-1} = \ast$, $\psi_{-2}$ induces a well-defined map
\[
 \varphi_{-2} \co  \Omega_{-1} \longrightarrow \overline{\Omega}_{-2}.
\]
More explicitly $\varphi_{-2}$ is given by
\[
\varphi_{-2}(x) = \left\{
\begin{array}{ll}
\psi_{-2}(x) & \mbox{if $x\in \overline{\Omega}_{-1}$} \\
{*} & \mbox{otherwise.}
\end{array}\right.
\]
From this description, it is easy to show $\psi_{-3}\varphi_{-2} = \ast$.

Inductively we obtain a map
\[
\varphi_{-n} \co  \Omega_{-n+1} \longrightarrow \overline{\Omega}_{-n}
\]
with $\psi_{-n-1}\varphi_{-n} = \ast$. Define $\Omega_{-n}$ to be the
mapping cone of $\varphi_{-n}$.

For any homology theory $h_*(-)$, the sequence of cofibrations
\eqref{cofibrations} induces an exact couple
\begin{eqnarray*}
 \raisebox{5pt}{$\scriptstyle R$}D^1_{-p,q} & = &
  \tilde{h}_{q}(\Omega_{-p}) \\
 \raisebox{5pt}{$\scriptstyle R$}E^1_{-p,q} & = &
  \tilde{h}_{q}(\overline{\Omega}_{-p})
\end{eqnarray*}
Rector proved the following theorem in \cite{Rector70}. 
\begin{theorem}
 \label{RectorIsEMSS}
 When $h_*(-)$ is the singular homology theory, the spectral sequence
 associated with the above exact couple is naturally isomorphic to the
 original Eilenberg--Moore spectral sequence constructed in
 \cite{Eilenberg-Moore66-1,Eilenberg-Moore66-2}.
\end{theorem}

On the other hand, in \cite{L.SmithEMSS,L.Smith70}, Larry Smith
introduced another construction.
The first step in Smith's approach is to consider a fibration $\smash{X
\rarrow{p} B}$ as an object in the category of spaces over $B$.

\begin{definition}
Let $B$ be an arbitrary space. A space over $B$ is a continuous map
$f\co  X \longrightarrow B$. A pointed space over $B$ is a pair of maps
\begin{eqnarray*}
 f & \co  & X \longrightarrow B \\
 s & \co  & B \longrightarrow X
\end{eqnarray*}
with $f\circ s = \mathrm{id}_B$. For simplicity, we denote this object by $(f,s)$.
Morphisms between these objects are obviously defined. 

If $f\co  X\longrightarrow B$ is a space over $B$, we sometimes denote the
 ``total space'' $X$ by $T(f)$. Similarly, for $(f,s)$ a pointed space
 over $B$, we use the notation $T(f) = X$ if $f\co  X\longrightarrow B$.

The category of spaces over $B$ is denoted by $\Spaces/B$. The category of
 pointed spaces over $B$ is denoted by $(\Spaces/B)_\ast$.

When $B = \ast$, we simply denote $\Spaces$ and $\Spaces_\ast$ for
$\Spaces/\ast$ and $(\Spaces/\ast)_\ast$, respectively. They are the usual
categories of spaces and  pointed spaces, respectively.
\end{definition}

Almost all important constructions and notions in $\Spaces$ or in
$\Spaces_\ast$ have analogies in $\Spaces/B$ or in
$(\Spaces/B)_\ast$. To be self-contained, we
record some of them used in the following section.

\begin{definition}
 For morphisms $\varphi_0,\varphi_1 \co  (f,s) \longrightarrow (g,t)$, a
 homotopy from $\varphi_0$ to $\varphi_1$ is a map
 \[
 \varphi \co  T(f)\times I \longrightarrow T(g)
 \]
 which fits into the following commutative diagrams, where $r\in [0,1]$.
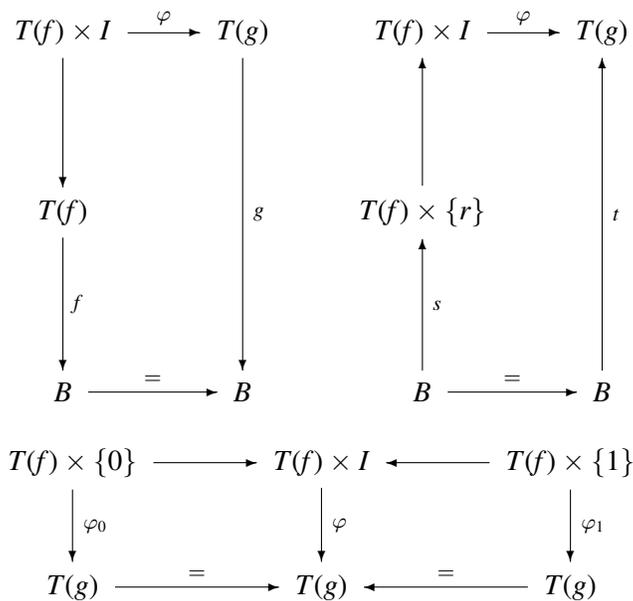
\[
 \begin{diagram}
  \node{T(f)\times I} \arrow{e,t}{\varphi} \arrow{s} \node{T(g)}
  \arrow[2]{s,r}{g} \node{T(f)\times I} \arrow{e,t}{\varphi} \node{T(g)}
  \\
  \node{T(f)} \arrow{s,r}{f} \node{} \node{T(f)\times \{r\}} \arrow{n}
  \node{} \\
  \node{B} \arrow{e,t}{=} \node{B} \node{B} \arrow{n,r}{s}
  \arrow{e,t}{=} \node{B} \arrow[2]{n,r}{t}
 \end{diagram}
\]
\[ \begin{diagram}
  \node{T(f)\times\{0\}} \arrow[2]{e} \arrow{s,r}{\varphi_0} \node{}
  \node{T(f)\times I} \arrow{s,r}{\varphi} \node{}
  \node{T(f)\times\{1\}} \arrow[2]{w} \arrow{s,r}{\varphi_1} \\
  \node{T(g)} \arrow[2]{e,t}{=} \node{} \node{T(g)} \node{} \node{T(g)}
  \arrow[2]{w,t}{=}
 \end{diagram}
\]
\end{definition}

\begin{definition}
 A morphism $\varphi \co  (f,s) \longrightarrow (g,t)$ in $(\Spaces/B)_\ast$ is
 called a cofibration in $(\Spaces/B)_\ast$ if it has the homotopy extension
 property with respect to the homotopy defined above.
\end{definition}

\begin{definition}
 For a morphism $\varphi \co  (f,s) \longrightarrow (g,t)$ in $(\Spaces/B)_\ast$,
 the (reduced) mapping cone of $\varphi$, denoted by 
 $(C(\varphi), S_C(\varphi))$, is defined as follows.
 \[
 T(C(\varphi)) = T(f)\times I \amalg T(g) \left/
 \begin{array}{ll}
  (x,0) \sim \varphi(x) & \mbox{for $x \in T(f)$} \\
  (s(b),r) \sim t(b) & \mbox{for $b \in B,r\in I$} \\
  (x,1) \sim (x',1) & \mbox{if $f(x) = f(x')$ for $x,x' \in T(f)$}
 \end{array}
 \right.
 \]
 The projection $C(\varphi) \co  T(C(\varphi)) \longrightarrow B$ is given
 by $f$ on $T(f)\times I$ and by $g$ on $T(g)$. The section
 $S_C(\varphi) \co  B\longrightarrow T(C(\varphi))$ is defined either by
 $s$ or $t$ which agree in $T(C(\varphi))$.
\end{definition}

\begin{definition}
 For an object $(f,s) \in (\Spaces/B)_\ast$ the (reduced) suspension of
 $(f,s) \in (\Spaces/B)_\ast$, denoted by $S(f,s) = (S(f),S(s))$,
 is defined by
 \[
 T(S(f)) = T(f)\times I \left/
 \begin{array}{ll}
  (x,0) \sim (x',0) & \mbox{if $f(x) = f(x')$ for $x,x'\in T(f)$} \\
  (x,1) \sim (x',1) & \mbox{if $f(x) = f(x')$ for $x,x'\in T(f)$} \\
  (s(b),r) \sim (s(b),r') & \mbox{for $b\in B$ and $r,r' \in I$}
 \end{array}
 \right.
 \]
\[\eqaligntop{
  S(f)(x,t) & =  f(x) \tag*{\hbox{and}}\cr
  S(s)(b) & =  (s(b),0).
 }\]
\end{definition}

\begin{definition}
 For objects $(f,s)$ and $(g,t)$ in $(\Spaces/B)_\ast$, the smash product of
 $(f,s)$ and $(g,t)$, denoted by $(f,s)\wedge (g,t) = 
(f\wedge_B g,s\wedge_B t)$, is defined by
\begin{multline*}
T(f\wedge_B g) = T(f)\times_B T(g) / (x,t(b)) \sim (s(b),y) 
\\\mbox{if $f(x) = g(y) = b$ for $x\in T(f), y\in T(g), b\in B$}
\end{multline*}
\[\eqaligntop{
(f\wedge_B g)(x,y) & =  f(x) = g(y) \tag*{\hbox{and}}\cr
(s\wedge_B t)(b) & =  (s(b),t(b))}
\]
\end{definition}

\begin{lemma}
\label{smash}
 If $(f,s) \rarrow{\varphi} (g,t)$ is a cofibration in
 $(\Spaces/B)_\ast$, then so is
 \[
 (h,u)\wedge (f,s) \rarrow{\mathrm{id}\mskip1mu\wedge\mskip1mu \varphi} (h,u)\wedge (g,t)
 \]
 with cofiber $(h,u)\wedge (C(\varphi),S_C(\varphi))$.
\end{lemma}

We fix notations for forgetful functors and their adjoints among the
above categories.

\begin{definition}
 Consider the functor given by forgetting sections
 \[
 F \co  (\Spaces/B)_\ast \longrightarrow \Spaces/B.
 \]
 Its adjoint is denoted by
 \[
 G \co  \Spaces/B \longrightarrow (\Spaces/B)_\ast
 \]
 which is given, on total spaces, by
 \[
 T(G(f)) = T(f) \amalg B.
 \]
 The section is defined to be the identity map into the second component
 $B$.
 \[
 \Gamma \co  \Spaces_\ast \longrightarrow (\Spaces/B)_\ast \leqno{\hbox{Let}}
 \]
 be the functor defined on the total spaces, by
 \[
 T(\Gamma(X)) = X\times B.
 \]
 The section is defined by the composition
 \[
 B = \{*\}\times B \hookrightarrow X\times B
 \]
 where $\ast$ is the base point of $X$. Its adjoint is the functor
 \[
 \Phi \co  (\Spaces/B)_\ast \longrightarrow \Spaces_\ast
 \]
 defined, for an object $(f,s)$, by
 \[
 \Phi(f,s) = T(f)/s(B).
 \]
\end{definition}

\begin{definition}
 For any nonnegative integer $n$, we denote $(S^n_B,s^n_B) = \Gamma(S^n)$.
\end{definition}

With these functors, we can describe the definition of homotopy in
$(\Spaces/B)_\ast$ more compactly.

\begin{lemma}
 For morphisms $\varphi_0,\varphi_1 \co  (f,s) \longrightarrow (g,t)$, a
 homotopy from $\varphi_0$ to $\varphi_1$ is a morphism
 \[
 \varphi \co  (f,s)\wedge_B \Gamma(I_+) \longrightarrow (g,t)
 \]
 with the following commutative diagram
 \[
 \begin{diagram}
  \node{(f,s)\wedge_B \Gamma(\{0\}_+)} \arrow{s,r}{\|} \arrow[2]{e}
  \node{} \node{(f,s)\wedge_B \Gamma(I_+)} \arrow{s,r}{\varphi} \node{}
  \node{(f,s)\wedge_B \Gamma(\{1\}_+)} \arrow[2]{w} \arrow{s,r}{\|} \\
  \node{(f,s)} \arrow[2]{e,t}{\varphi_0} \node{} \node{(g,t)} \node{}
  \node{(f,s)} \arrow[2]{w,t}{\varphi_1}
 \end{diagram}
 \]
\end{lemma}

The functor $\Phi$ has a very good property.

\begin{lemma}
\label{CofiberPreserving}
 For any cofibration 
 \[
 (f_0,s_0) \longrightarrow (f_1,s_1) \longrightarrow (f_2,s_2)
 \]
 in $(\Spaces/B)_\ast$,
 \[
 \Phi(f_0,s_0) \longrightarrow \Phi(f_1,s_1) \longrightarrow \Phi(f_2,s_2)
 \]
 is a cofibration in $\Spaces_\ast$.
\end{lemma}

Thanks to this lemma, any (reduced) homology theory on $\Spaces_\ast$
naturally extends to $(\Spaces/B)_\ast$.

\begin{definition}
 Let  $\tilde{h}_\ast(-)$ be any homology theory on $\Spaces_\ast$ and 
  \hbox{$(f,s) \co  X \longrightarrow B$} any
 object in $(\Spaces/B)_\ast$. Define
 \[
 h^B_*(f,s) = \tilde{h}_*\circ\Phi(f,s)= \tilde{h}_*(X/s(B)).
 \]
 Thus we have a covariant functor
 \[
 h^B_*(-) \co  (\Spaces/B)_* \longrightarrow \gAbels,
 \]
 where $\gAbels$ denotes the category of graded Abelian groups. The functor
 $h^B_*(-)$ is referred to as the homology theory on
 $(\Spaces/B)_\ast$ associated with $\tilde{h}_*(-)$.
\end{definition}

\begin{corollary}
For any cofibration in $(\Spaces/B)_\ast$
\[
(f_0,s_0) \longrightarrow (f_1,s_1) \longrightarrow (f_2,s_2)
\]
and a homology theory $\tilde{h}_*(-)$ on $\Spaces_\ast$, we have a long
 exact sequence:
\[
\cdots \longrightarrow h^B_*(f_0,s_0) \longrightarrow h^B_*(f_1,s_1)
 \longrightarrow h^B_*(f_2,s_2) \rarrow{\partial} h^B_{*-1}(f_0,s_0)
 \longrightarrow \cdots
\]
\end{corollary}

Now we are ready to recall the construction of a cobar-type
Eilenberg--Moore spectral sequence by Larry Smith. His idea is to
construct a spectral sequence out of a ``filtration'' in the category
$(\Spaces/B)_\ast$, ie\ display. Smith made a lot of assumptions on
(co)homology theory and the space $B$ in his paper
\cite{L.Smith70}. Most of his assumptions are for the
existence of a display and the convergence of the spectral
sequence. Since our purpose is to show Rector's construction and the
gravity filtration on $\Omega^n\Sigma^n X$ give rise to displays, we do
not need these assumptions. What we really need is the following.

\begin{assumption}
 Throughout the rest of this section, $\tilde{h}_*(-)$ denotes a (reduced)
 multiplicative homology theory. We also assume that the external
 product
 \[
 \tilde{h}_*(B_+) \otimes_{h_*} \tilde{h}_*(X) \longrightarrow
 \tilde{h}_*(B_+\wedge X)
 \]
 is an isomorphism for any pointed space $X$. This condition is
 satisfied, for example, if $\tilde{h}_*(B_+)$ is $h_*$--flat.
\end{assumption}

This condition is necessary for the following definition.

\begin{definition}
 Given any pointed space $(f,s)$ over $B$, define a structure of left
 $h_*^B(S_B^0,s_B^0)$--comodule on $h_*^B(f,s)$ by the composition
 \begin{eqnarray*}
  h_*^B(f,s) \cong \tilde{h}_*(\Phi(f,s)) \rarrow{\Delta_*}
   \tilde{h}_*(\Phi(f,s)\wedge \Phi(f,s))& \rarrow{\Phi(*)\wedge \mathrm{id}} &
   \tilde{h}_*(\Phi(S^0_B,s^0_B)\wedge \Phi(f,s)) \\
  & = & \tilde{h}_*(B_+\wedge \Phi(f,s)) \\
  & \cong & \tilde{h}_*(B_+)\otimes_{h_*} \tilde{h}_*(\Phi(f,s)) \\
  & = & h_*^B(S_B^0,s_B^0)\otimes_{h_*} h_*^B(f,s).
 \end{eqnarray*}
 Similarly $h_*^B(f,s)$ also has a structure of right
 $h_*^B(S^0_B,s^0_B)$--comodule.
\end{definition}

\begin{definition}
 Let $(f,s)$ be an object in $(\Spaces/B)_\ast$. An $h_*^B$--display of
 $(f,s)$ is a sequence of cofibrations
 \[
 \begin{array}{ccccc}
  (f,s) & \rarrow{\alpha_0} & (h_0,u_0) & \rarrow{\beta_{-1}} &
   (f_{-1},s_{-1}) \\
  (f_{-1},s_{-1}) & \rarrow{\alpha_{-1}} & (h_{-1},u_{-1}) &
   \rarrow{\beta_{-2}} & (f_{-2},s_{-2}) \\
  & & \vdots & & \\
  (f_{-i},s_{-i}) & \rarrow{\alpha_{-i}} & (h_{-i},u_{-i}) &
   \rarrow{\beta_{-i-1}} & (f_{-i-1},s_{-i-1}) \\
  & & \vdots & & 
 \end{array}
\]
 satisfying the following two conditions.
 \begin{enumerate}
  \item $h_*^B(h_{-i},u_{-i})$ is a flat $h_*$--module and an injective
    $h_*^B(S^0_B,s^0_B) \cong h_*(B_+)$--comodule for each $i$.
  \item ${\alpha_{-i}}_* \co  h_*^B(f_{-i},s_{-i}) \longrightarrow
    h_*^B(h_{-i},u_{-i})$ is a monomorphism.
 \end{enumerate}
\end{definition}

Suppose $\{(f_{-i},s_{-i}),(h_{-i},u_{-i})\}$ is an $h_*^B$--display of
$(f,s)$. Let $(g,t)$ be another pointed space over $B$. By \fullref{smash}, smashing with $(g,t)$ preserves cofibrations and we have
cofibrations:
\[
\begin{array}{ccccc}
 (g,t)\wedge (f,s) & \xrightarrow{\substack{\hphantom{\mathrm{id}\wedge\mskip1mu \alpha_{-1}}\\\mathrm{id}\wedge\mskip1mu \alpha_0}} & (g,t)\wedge (h_0,u_0)
  & \xrightarrow{\substack{\hphantom{\mathrm{id}\wedge\mskip1mu \beta_{-i-1}}\\\mathrm{id}\wedge\mskip1mu \beta_{-1}}} & (g,t)\wedge (f_{-1},s_{-1}) \\
 (g,t)\wedge (f_{-1},s_{-1}) & \xrightarrow{\mathrm{id}\wedge\mskip1mu \alpha_{-1}} &
  (g,t)\wedge (h_{-1},u_{-1}) & \xrightarrow{\substack{\hphantom{\mathrm{id}\wedge\mskip1mu \beta_{-i-1}}\\\mathrm{id}\wedge\mskip1mu \beta_{-2}}} &
  (g,t)\wedge (f_{-2},s_{-2}) \\
 & & \vdots & & \\
 (g,t)\wedge (f_{-i},s_{-i}) & \xrightarrow{\substack{\hphantom{\mathrm{id}\wedge\mskip1mu \alpha_{-1}}\\\mathrm{id}\wedge\mskip1mu \alpha_{-i}}} &
  (g,t)\wedge (h_{-i},u_{-i}) & \xrightarrow{\mathrm{id}\wedge\mskip1mu \beta_{-i-1}} &
  (g,t)\wedge (f_{-i-1},s_{-i-1}) \\
 & & \vdots & & 
\end{array}
\]
Thus we obtain an exact couple by applying $h_*^B(-)$.

\begin{definition}
 \label{DefofEMSS}
 Define
 \begin{eqnarray*}
  D^1_{-p,q}((f,s),(g,t)) & = & h_{-p}^B((f_{-p},s_{-p})\wedge(g,t)) \\
  E^1_{-p,q}((f,s),(g,t)) & = & h_{-p}^B((h_{-p},u_{-p})\wedge(g,t)).
 \end{eqnarray*}
 The cobar spectral sequence or the K\"unneth spectral sequence defined
 by an $h_*^B$--display
 \[
 \{(f_{-i},s_{-i}),(h_{-i},u_{-i})\}
 \]
 is the spectral sequence denoted
 $\{E^r_{*,*}((f,s),(g,t))\}$ associated with this exact couple.
\end{definition}

In order to identify the $E^2$--term, it is important to use a special
kind of display. To see this, let $(f,s),(g,t) \in (\Spaces/B)_\ast$ and
\[
\{(f_{-i},s_{-i}),(h_{-i},u_{-i})\}
\]
be an $h_*^B$--display. Since $h_*^B(h_{-i},u_{-i})$ is flat over $h_*$,
\begin{align*}
\tilde{h}_*(\Phi(h_{-i},u_{-i})\wedge \Phi(g,t)) &\cong
\tilde{h}_*(\Phi(h_{-i},u_{-i}))\otimes_{h_*} \tilde{h}_*(\Phi(g,t)) \\ &=
h_*^B(h_{-i},u_{-i})\otimes_{h_*} h_*^B(g,t).
\end{align*}
Note that the following composition is trivial:
\begin{eqnarray*}
h_*^B((h_{-i},u_{-i})\wedge (g,t)) & \longrightarrow &
 \tilde{h}_*(\Phi(h_{-i},u_{-i})\wedge \Phi(g,t)) \\
& \cong & h_*^B(h_{-i},u_{-i})\otimes_{h_*} h_*^B(g,t) \\
& \xrightarrow{\psi\otimes \mathrm{id}-\mathrm{id}\otimes \varphi} &
 h_*^B(h_{-i},u_{-i})\otimes_{h_*} h_*^B(S^0_B,s^0_B) \otimes_{h_*}
 h_*^B(g,t).
\end{eqnarray*}
\[\eqaligntop{
 \Delta_*  \co   h_*^B(S^0_B,s^0_B) &\longrightarrow
 h_*^B(S^0_B,s^0_B)\otimes_{h_*} h_*^B(S^0_B,s^0_B) \tag*{\hbox{where}}\cr
 \psi  \co   h_*^B(h_{-i},u_{-i}) &\longrightarrow
 h_*^B(h_{-i},u_{-i})\otimes_{h_*} h_*^B(S^0_B,s^0_B) \cr
 \varphi  \co  h_*^B(g,t)& \longrightarrow h_*^B(S^0_B,s^0_B)\otimes_{h_*}
 h_*^B(g,t)}
\]
are coalgebra and comodule structure maps, respectively. Thus we obtain
a map
\[
h_*^B((h_{-i},u_{-i})\wedge (g,t)) \longrightarrow \Ker(\psi\otimes
\mathrm{id}-\mathrm{id}\otimes \varphi).
\]
By the definition of cotensor product
\[
\Ker(\psi\otimes \mathrm{id} - \mathrm{id} \otimes \varphi) =
h_*^B(h_{-i},u_{-i})\Box_{h_*^B(S^0_B,s^0_B)}h_*^B(g,t).
\]
It is convenient to assume the resulting map
\[
\Psi \co  h_*^B((h_{-i},u_{-i})\wedge (g,t)) \longrightarrow
h_*^B(h_{-i},u_{-i})\Box_{h_*^B(S^0_B,s^0_B)}h_*^B(g,t)
\]
is an isomorphism.

\begin{definition}
Let $(f,s)$ and $(g,t)$ be pointed spaces over $B$. An $h_*^B$--display
of $(f,s)$
\[
\{(f_{-i},s_{-i}),(h_{-i},u_{-i})\}
\]
is said to be injective with respect to $(g,t)$ if
\[
\Psi \co  h_*^B((h_{-i},u_{-i})\wedge (g,t)) \longrightarrow
 h_*^B(h_{-i},u_{-i})\Box_{h_*^B(S^0_B,s^0_B)} h_*^B(g,t)
\]
is an isomorphism.
\end{definition}

The following lemma is immediate from the definition.

\begin{lemma}
 If $\{E^r_{*,*}((f,s)\wedge(g,t)),d^r\}$ is a cobar spectral sequence
 defined by an injective $\smash{h_*^B}$--display of $(f,s)$ with respect to
 $(g,t)$, then we have
 \begin{eqnarray*}
  E^2_{*,*}((f,s)\wedge (g,t)) & \cong &
   \Cotor^{h_*^B(S^0_B,s^0_B)}(h_*^B(f,s),h_*^B(g,t)) \\
  & = & \Cotor^{\tilde{h}_*(B_+)}(\tilde{h}_*(\Phi(f,s)),\tilde{h}_*(\Phi(g,t))).
 \end{eqnarray*}
\end{lemma}

This lemma suggests that cobar spectral sequences defined by using
injective displays are isomorphic to each other from the $E^2$--term
on. In fact, this is the case.

\begin{theorem}
\label{Uniqueness}
Let $(f,s)$ and $(g,t)$ be pointed spaces over $B$. Let
\[
\begin{array}{c}
\{(f_{-i},s_{-i}),(h_{-i},u_{-i})\} \\
\{({f'}_{-i},{s'}_{-i}),({h'}_{-i},{u'}_{-i})\}
\end{array}
\]
be injective $h_*^B$--displays for $(f,s)$ with respect to $(g,t)$. Let
 $\{E^r\}$ and $\{{}'E^r\}$ be the cobar spectral sequences defined by
 the first and the second display, respectively. Then we have an
 isomorphism of spectral sequences for $r \ge 2$,
\[
E^r \cong {}'E^r.
\]
\end{theorem}

\begin{proof}
 See pp\,119--120 of \cite{L.Smith70}. Smith proved this fact by finding
 an intermediate display
 \[
 \{(\bar{f}_{-i},\bar{s}_{-i}),(\bar{h}_{-i},\bar{u}_{-i})\}
 \]
 and maps
 \[
 \{(f_{-i},s_{-i}),(h_{-i},u_{-i})\} \longrightarrow
 \{(\bar{f}_{-i},\bar{s}_{-i}),(\bar{h}_{-i},\bar{u}_{-i})\}
 \longleftarrow \{{f'}_{-i},{s'}_{-i}),({h'}_{-i},{u'}_{-i})\}.
 \]
 The existence of such a display in our case is essentially proved in
 pp\,112--113 of the same paper.
\end{proof}

\section{Comparing spectral sequences}
\label{Comparison}

Let us compare Rector's construction for
\[
 \ast \longrightarrow \Omega^{n-1}\Sigma^n X \longleftarrow
 P\Omega^{n-1}\Sigma^n X
\]
with the construction by the gravity filtration. We proved that the
$E^1$--terms of spectral sequences are isomorphic
as chain complexes in \fullref{Decompose}. In order to show that these
spectral sequences are isomorphic from the $E^2$--term, it is enough to
show that they both give rise to injective displays.

We first prove the following general fact.

\begin{theorem}
 \label{RectorandSmith}
 Consider the pullback diagram:
 \[
 \begin{diagram}
  \node{Y\times_B X} \arrow{s} \arrow{e} \node{X} \arrow{s,r}{p} \\
  \node{Y} \arrow{e,t}{f} \node{B}
 \end{diagram}
 \]
 
 If $\tilde{h}_*(B_+)$ is $h_\ast$--flat and $p$ is a fibration, the
 spectral sequence induced from Rector's geometric cobar construction
 for this pullback diagram is isomorphic to Smith's spectral sequence
 from the $E^2$--term on.
\end{theorem}

Let $p\co  X \longrightarrow B$ be a fibration. Using the functor
\[
 G \co  \Spaces/B \longrightarrow (\Spaces/B)_\ast
\]
we obtain an object $G(p)$ in $(\Spaces/B)_\ast$.  Let
$f \co  Y \longrightarrow B$ be a continuous map. In the following we
construct a display for $G(p)$
\[
\begin{array}{ccccc}
 G(p) & \longrightarrow & (\omega_{-1},u_{-1}) & \longrightarrow &
  (p_{-1},s_{-1}) \\
 (p_{-1},s_{-1}) & \longrightarrow & (\omega_{-2},u_{-2}) &
  \longrightarrow & (p_{-2},s_{-2}) \\
 & & \vdots & & 
\end{array}
\]
so that the filtration on $Y\times_B X$ induced by this display is the
same as Rector's cosimplicial construction.

Recall that Rector's spectral sequence is induced from the cofibration
sequences
\[
\begin{array}{ccccc}
 & & \vdots & & \\
 \Omega_{-\ell+1} & \rarrow{\varphi_{-\ell}} & \overline{\Omega}_{-\ell}
  & \longrightarrow & \Omega_{-\ell} \\
 & & \vdots & & 
\end{array}
\]
while Smith's spectral sequence is induced from cofibration sequences:
\[
\begin{array}{ccccc}
 \Phi(G(f)\wedge_B G(p)) & \!\!{\longrightarrow}\!\! & \Phi(G(f) \wedge_B
  (\omega_{-1},u_{-1})) & \!\!{\longrightarrow}\!\! & \Phi(G(f) \wedge_B
  (p_{-1},s_{-1})) \\
 \Phi(G(f) \wedge_B (p_{-1},s_{-1})) & \!\!{\longrightarrow}\!\! & \Phi(G(f)
  \wedge_B (\omega_{-2},u_{-2})) & \!\!{\longrightarrow}\!\! & \Phi(G(f) \wedge_B
  (p_{-2},s_{-2})) \\
 & & \vdots & & 
\end{array}
\]
But it is not difficult to find $(p_{-\ell},s_{-\ell})$ and
$(\omega_{-\ell},u_{-\ell})$ with
\begin{eqnarray*}
\Omega_{-\ell} & \simeq & \Phi(G(f)\wedge_B(p_{-\ell},s_{-\ell})) \\
\overline{\Omega}_{-\ell} & \simeq & \Phi(G(f)\wedge_B(\omega_{-\ell},u_{-\ell})).
\end{eqnarray*}
It is natural to expect that the process of constructing such a display
is very similar to that of the cofibration sequences of Rector's.
We first define $(\omega_{-\ell},u_{-\ell})$ and then, inductively,
$(p_{-\ell},s_{-\ell})$.

\begin{definition}
 For $\ell \ge 0$, define
 \[
 T(\omega_{-\ell}) = B\times (B^{\ell}\times
 X/\Ima\delta^1\cup\cdots\cup\Ima\delta^{\ell})
 \]
 where the maps $\delta^i$ for $i= 1,\cdots,\ell$ are the maps in the
 geometric cobar construction with $Y = \ast$.  The map
 \[
 \omega_{-\ell} \co  T(\omega_{-\ell}) \longrightarrow B
 \]
 is just the projection onto the first factor. The section
 \[
 u_{-\ell} \co  B \longrightarrow T(\omega_{-\ell})
 \]
 is defined by $u_{-\ell}(b) = (b,*)$.
\end{definition}

As is the case of the geometric cobar construction, ``$\delta^0$''
induces a map
\begin{align*}
\psi_{-\ell} \co  (\omega_{-\ell},u_{-\ell}) &\longrightarrow
(\omega_{-\ell-1},u_{-\ell-1})
\\
\psi_{-\ell}(b;b_1,\cdots,b_{\ell},x) &= (b;b,b_1,\cdots,b_{\ell},x).\mskip90mu \tag*{\hbox{defined by}}
\end{align*}
We also have $\psi_{-\ell-1}\psi_{-\ell} = \ast$.

We need to check the following lemma.

\begin{lemma}
 Suppose $f\co  Y \longrightarrow B$ is surjective. Then for $\ell \ge 0$,
 we have a homeomorphism
 \[
 \overline{\Omega}_{-\ell} \cong \Phi(G(f)\wedge (\omega_{-\ell},u_{-\ell}))
 \]
 making the following diagram commutative.
 \begin{equation}
  \begin{diagram}
   \node{\overline{\Omega}_{-\ell}} \arrow{s,r}{\psi_{-\ell}}
   \arrow[2]{e,t}{\cong} \node{} \node{\Phi(G(f)\wedge
   (\omega_{-\ell},u_{-\ell}))} \arrow{s,r}{\Phi(1\wedge \psi_{-\ell})}
   \\
   \node{\overline{\Omega}_{-\ell-1}} \arrow[2]{e,t}{\cong} \node{}
   \node{\Phi(G(f)\wedge (\omega_{-\ell-1},u_{-\ell-1}))}
  \end{diagram}
  \label{Phi}
 \end{equation}
\end{lemma}

\begin{proof}
 By the definition of the smash product in $\Spaces_*/B$.
 \begin{eqnarray*}
  T(G(f)\wedge (\omega_{-\ell},u_{-\ell})) & = & T(G(f))\times_B
   T(\omega_{-\ell})/ (y,u_{-\ell}(b))\sim (b,x) \ \ \mbox{if $f(y) =
   b)$} \\
  & = & (Y\amalg B)\times_B B\times (B^{\ell}\times
   X/\Ima\delta^1\cup\cdots\cup\Ima\delta^{\ell})/\\
  & & \hspace{140pt}(y,f(y),*)\sim
   (f(y),x) \\
  & = & Y\times (B^{\ell}\times
   X/\Ima\delta^1\cup\cdots\cup\Ima\delta^{\ell}) \\
  & & \amalg B\times (B^{\ell}\times
   X/\Ima\delta^1\cup\cdots\cup\Ima\delta^{\ell}) /(y,*) \sim (f(y),x).
 \end{eqnarray*}
 Thus
\[\eqaligntop{
&\Phi(T(G(f)\wedge (\omega_{-\ell},u_{-\ell}))) \cr
  & =  \left.\frac{Y\times (B^{\ell}\times
     X/\Ima\delta^1\cup\cdots\cup\Ima\delta^{\ell}) \amalg B\times
     (B^{\ell}\times
     X/\Ima\delta^1\cup\cdots\cup\Ima\delta^{\ell})}{(y,*) \sim
     (f(y),x)} \!\!\right/\!\! B\!\times\! * \cr
  & =  Y\times (B^{\ell}\times
   X/\Ima\delta^1\cup\cdots\cup\Ima\delta^{\ell})/Y\times * \cr
  & =  Y\times B^{\ell}\times X/\Ima\delta^1\cup\cdots\cup\Ima\delta^{\ell} \cr
  & =  \overline{\Omega}_{-\ell}.
}\]
 With this description of the homeomorphism, it is easy to check the
 diagram \eqref{Phi} commutes.
\end{proof}

We can deform any continuous map $f\co  Y \longrightarrow B$ into a
surjective map up homotopy. This does not change the homotopy type of
the pullback by $f$ since $p$ is a fibration.

In the following, we always assume that $f \co  Y \longrightarrow B$ is
surjective.

\begin{corollary}
For each $\ell$, if $\tilde{h}_*(B_+)$ is $h_*$--flat, we have an
isomorphism
\[
 h_*^B(G(f)\wedge (\omega_{-\ell},u_{-\ell})) \cong
 h_*^B(G(f))\Box_{h_*^B(S^0_B,s^0_B)} h_*^B(\omega_{-\ell},u_{-\ell}).
\]
Thus the sequence of cofibrations
\[
\{ (p_{-\ell+1},s_{-\ell+1}) \longrightarrow (\omega_{-\ell},u_{-\ell})
 \longrightarrow (p_{-\ell},s_{-\ell}) \}_{\ell = 1,2,\cdots}
\]
just constructed is an injective $h_*^B$--display of $G(p)$ with respect
 to $G(f)$.
\end{corollary}

\begin{proof}
Rector proved in \cite{Rector70} that
\begin{eqnarray*}
\tilde{h}_*(\overline{\Omega}_{-\ell}) & \cong &
 \tilde{h}_*(Y_+)\Box_{\tilde{h}_*(B_+)}
 \tilde{h}_*(\Phi(\omega_{-\ell},u_{-\ell})) \\
& \cong & h_*^B(G(f))\Box_{h_*^B(S^0_B,s^0_B)}
 h_*^B(\omega_{-\ell},u_{-\ell}).
\end{eqnarray*}
Therefore the display
 $\{(p_{-\ell},s_{-\ell}),(\omega_{-\ell},u_{-\ell})\}$ is an injective
 display.
\end{proof}

Let us prove that the spectral sequence induced by the above display is
identical to Rector's spectral sequence.

Let $\alpha_{-1} \co  G(p) \longrightarrow (\omega_{-1},u_{-1})$ be the map
in $(\Spaces/B)_\ast$ defined on the total space by the following composition.
\begin{align*}
T(G(p)) = X \amalg B \xrightarrow{p\times \mathrm{id}_X\amalg \mathrm{id}_B} B\times X \amalg B
&= B\times X \amalg B\times \{*\}\\ &= B\times X/\phi = T(\omega_0)
\rarrow{\psi_{-1}} T(\omega_{-1}).
\end{align*}
Let $(p_{-1},s_{-1})$ be the cofiber of $\alpha_{-1}$. Since
$\psi_{-1}\psi_{-2} = \ast$, $\psi_{-2}$ induces a map
\[
\alpha_{-2} \co  (p_{-1},s_{-1}) \longrightarrow (\omega_{-2},u_{-2})
\]
with $\psi_{-3}\alpha_{-2} = \ast$. Inductively on $\ell$, we can define 
$(p_{-\ell+1},s_{-\ell+1})$ and a map
\[
\alpha_{-\ell} \co  (p_{-\ell+1},s_{-\ell+1}) \longrightarrow
(\omega_{-\ell},u_{-\ell})
\]
\[
\psi_{-\ell-1}\alpha_{-\ell} = \ast. \leqno{\hbox{with}}
\]
By the commutativity of the diagram \eqref{Phi} and \fullref{CofiberPreserving}, we have an equivalence of cofibrations
\[
\begin{diagram}
\node{\Omega_{-\ell+1}} \arrow{s} \arrow[2]{e} \node{}
 \node{\Phi(G(f)\wedge (p_{-\ell+1},s_{-\ell+1}))} \arrow{s} \\
\node{\overline{\Omega}_{-\ell}} \arrow{s} \arrow[2]{e} \node{}
 \node{\Phi(G(f)\wedge (\omega_{-\ell},u_{-\ell}))} \arrow{s} \\
\node{\Omega_{-\ell}} \arrow[2]{e} \node{} \node{\Phi(G(f)\wedge
 (p_{-\ell},s_{-\ell}))}
\end{diagram}
\]
This completes the proof of \fullref{RectorandSmith}.

It remains to show that the gravity filtration gives rise to an
injective display.
Before we investigate the gravity spectral sequence, we consider a
more general situation. Let
\[
\begin{diagram}
\node{F} \arrow{e} \arrow{s} \node{E} \arrow{s,r}{p} \\
\node{\{*\}} \arrow{e} \node{B.}
\end{diagram}
\]
be a pullback diagram of pointed spaces. The construction of the
Eilenberg--Moore spectral sequence in \fullref{OverB} is based on the
notion of display which is a filtration on $G(p)$ in the category of
pointed spaces over $B$. As we have seen in \fullref{OverB}, a display
induces a stable decreasing filtration on $F$, and hence a spectral
sequence abutting to the homology of $F$. In some cases, however, the
total space $E$ itself has a stable decreasing filtration which induces
a stable decreasing filtration on $F$. In fact, under some conditions, a
decreasing filtration on $E$ defines a display for $G(p)$ which induces
the same stable filtration on $F$ as the one induced by the filtration
on $E$. To be more precise, consider the following data and conditions.
\begin{enumerate}
 \item A pullback diagram of pointed spaces
      \[
      \begin{diagram}
       \node{F} \arrow{e} \arrow{s} \node{E} \arrow{s,r}{p} \\
       \node{\{*\}} \arrow{e,t}{\iota_0} \node{B}
      \end{diagram}
      \]
 \item A decreasing filtration on $E$
      \[
      \cdots \subset F_{-q-1}E \subset F_{-q}E \subset \cdots \subset
       F_{-1}E \subset F_0E = E
      \]
      in which each inclusion
      \[
      F_{-q-1}E \xrightarrow{i_{-q-1}} F_{-q}E
      \]
      is a pointed cofibration
 \item In the induced decreasing filtration on $F$
      \[
      \cdots \subset F_{-q-1}F \subset F_{-q}F \subset \cdots \subset
       F_{-1}F \subset F_0F = F,
      \]
      where $F_{-q}F = F_{-q}E\cap F$, each inclusion is also a
      pointed cofibration
\end{enumerate}

\begin{proposition}
 \label{FilteredFibration}
 Given the above data, let $p_{-q} \co  F_{-q}E \longrightarrow B$ be the
 restriction of $p$ on $F_{-q}E$. If
 \[
 \tilde{h}_*(F_{-q}E/F_{-q-1}E) \cong \tilde{h}_*(F_{-q}F/F_{-q-1}F)
 \otimes_{h_*} \tilde{h}_*(B_+)
 \]
 as $\tilde{h}_*(B_+)$--comodules, then
 \[
 \begin{array}{ccccc}
  G(p) & \longrightarrow & C(G(i_{-1})) & \longrightarrow & \Sigma G(p_{-1}) \\
  \Sigma G(p_{-1}) & \longrightarrow & \Sigma C(G(i_{-2})) &
   \longrightarrow & \Sigma^2 G(p_{-2}) \\
  & & \vdots & & \\
  \Sigma^q G(p_{-q}) & \longrightarrow & \Sigma^q C(G(i_{-q-1})) &
   \longrightarrow & \Sigma^{q+1} G(p_{-q-1}) \\
  & & \vdots & &
 \end{array}
\]
 is an injective $h_*^B$--display of $G(p)$ with respect to
 $G(\iota_0)$. The stable filtration on $F$ induced by this display
 coincides with the one induced directly from the filtration on $E$.
\end{proposition}

\begin{proof}
 By the construction, $\{\Sigma^q G(p_{-q}),\Sigma^q C(G(i_{-q-1}))\}$ is
 a display. Consider the map of cofibrations:
\[
 \begin{diagram}
  \node{\Phi(G(p_{-q-1}))} \arrow{s,l}{\|} \arrow[2]{e} \node{}
  \node{\Phi(G(p_{-q}))} \arrow[2]{e} \arrow{s,l}{\|} \node{}
  \node{\Phi(C(G(i_{-q-1})))} \arrow{s,..} \\
  \node{F_{-q-1}E_+} \arrow[2]{e} \node{} \node{F_{-q}E_+} \arrow[2]{e}
  \node{} \node{F_{-q}E/F_{-q-1}E}
 \end{diagram}
\]
Since the left and the middle vertical map are homotopy equivalences, the
 induced map in the right is also a homotopy equivalence. Thus
\[
\Phi(C(G(i_{-q-1}))) \simeq F_{-q}E/F_{-q-1}E.
\]
 On the other hand, since
\[
 F_{-q}E \times_B \{*\} = F_{-q}F,
\]
\[
G(p_{-q})\wedge G(\iota_0) = G(p_{-q}\circ j_{-q}),\leqno{\hbox{we have}}
\]
where $j_{-q} \co  F_{-q}F \hookrightarrow F_{-q}E$ is the inclusion.
Therefore
 \begin{eqnarray*}
  C(G(i_{-q-1}))\wedge G(\iota_0) & \simeq & C\left(G(p_{-q-1})\wedge
   G(\iota_0) \xrightarrow{G(i_{-q-1})} G(p_{-q-1})\wedge
   G(\iota_0)\right) \\
  & \simeq & C\left(G(p_{-q-1}\circ j_{-q}) \rarrow{} G(p_{-q}\circ j_{-q})\right).
 \end{eqnarray*}
 Thus the following chain of equalities holds:
\begin{eqnarray*}
 h_*^B(\Sigma^q C(G(i_{-q-1}))\wedge G(\iota_0)) & = & h_*^B(\Sigma^q
 C(G(p_{-q-1}\circ j_{-q}) \rarrow{} G(p_{-q}\circ j_{-q}))) \\
& = & \tilde{h}_*(\Sigma^q(F_{-q}F/F_{-q-1}F)) \\
& \cong & \tilde{h}_*(\Sigma^q(F_{-q}E/F_{-q-1}E))\Box_{\tilde{h}_*(B_+)} h_* \\
& = & h_*^B(\Sigma^q G(p_{-q}))\Box_{h_*^B(S^0_B,s^0_B)} h_*^B(G(*)) 
\end{eqnarray*}
 This proves that $\{\Sigma^q G(p_{-q}),\Sigma^q C(G(i_{-q-1}))\}$ is
 an injective $h_*^B$--display of $G(p)$ with respect to $G(\ast)$.
\end{proof}

\begin{corollary}
\label{IsotoEMSS}
 Under the assumption in the above proposition, the spectral sequence
 defined by the filtration on $F$
 \[
 E^1_{-q,*} = \tilde{h}_*(F_{-q}F/F_{-q-1}F) \Longrightarrow \tilde{h}_*(F)
 \]
 is isomorphic to the classical Eilenberg--Moore spectral sequence for the
 pullback diagram
 \[
 \begin{diagram}
  \node{F} \arrow{e} \arrow{s} \node{E} \arrow{s,r}{p} \\
  \node{\ast} \arrow{e} \node{B}
 \end{diagram}
 \]
 from the $E^2$--term on.
\end{corollary}

Let us apply this fact to the path-loop fibration,
\[
\Omega^n\Sigma^n X \longrightarrow P\Omega^{n-1}\Sigma^n X
\longrightarrow \Omega^{n-1}\Sigma^n X,
\]
namely the pullback diagram
\begin{equation}
\begin{diagram}
 \node{\Omega^n\Sigma^n X} \arrow[2]{e} \arrow{s} \node{}
 \node{P\Omega^{n-1}\Sigma^n X} \arrow{s} \\
 \node{\{*\}} \arrow[2]{e} \node{} \node{\Omega^{n-1}\Sigma^n X}
\end{diagram}
\label{PathLoop}
\end{equation}
or its little cube model
\[
 \begin{diagram}
  \node{C_n(X)} \arrow{s} \arrow{e} \node{E_n(CX,X)} \arrow{s} \\
  \node{\ast} \arrow{e} \node{C_{n-1}(\Sigma X)}
 \end{diagram}
\]
due to May \cite{May72}. Let us recall the definition of $E_n(CX,X)$.

\begin{definition}
 Let $(Y,B)$ be a pointed pair. We define a subspace $\mathcal{E}_n(j;Y,B)$ of
 $\Cu_{n}(j)\times Y^j$ as follows.
 \begin{eqnarray*}
  (c_1,\cdots,c_j;y_1,\cdots,y_j) \in \mathcal{E}_n(j;Y,B) &
   \Longleftrightarrow & \textrm{ if $y_k\not\in B$ then $c_k$ can be
   extended} \\
  & & \textrm{ to the right,}
 \end{eqnarray*}
 where by ``extend to the right'', we mean the following: for a cube 
 $$c = (f_1,\cdots,f_n) \co  (-1,1)^n\longrightarrow (-1,1)^n,$$ let
 $f_1^{\prime}(t)$ be the interval with $f_1^{\prime}(-1) = f_1(-1)$ and
 $f_1^{\prime}(1) =1$. Let $\tilde{c} = (f_1^{\prime}, f_2, {\cdots}, f_n)$. 
 For ${\bf c} = (c_1,\cdots,c_j) \in \Cu_{n}(j)$, we say $c_k$, for 
 $1\le k\le j$, can be extended to the right if the image of
 $\smash{\tilde{c}_k}$ does not intersect with the images of other cubes, 
 $\Ima \smash{\tilde{c}_k} \cap \Ima c_l = \phi$ for $l \neq k$.

 By restricting the defining relation of $C_n(Y)$ to
 $\amalg \Cu_{n}(j)\times_{\Sigma_j} Y^j$, we define
 \[
 E_n(Y,B) = (\coprod_j \mathcal{E}_n(j;Y,B)/\Sigma_j)/\sim.
 \]
 For $k>0$, we define
 \begin{align*}
 \mathcal{F}_k E_n(Y,B) &= \Ima (\coprod_{j=1}^{k}
 \mathcal{E}_n(j;Y,B)/\Sigma_j \longrightarrow E_n(Y,B))
 \\
 \widetilde{\mathcal{E}}_n(j;Y,B) &= \mathcal{F}_j
 E_n(Y,B)/\mathcal{F}_{j-1} E_n(Y,B).\tag*{\hbox{and}}
 \end{align*}
\end{definition}

These constructions have the following properties.

\begin{theorem}
 Under the same condition as above, the sequence
 \[
 C_n(X) \longrightarrow E_n(CX,X) \longrightarrow C_{n-1}(\Sigma X)
 \]
 is a quasifibration which is weakly homotopy equivalent to the
 path-loop fibration
 \[
 \Omega^n\Sigma^n X \longrightarrow P\Omega^{n-1}\Sigma^n X
 \longrightarrow \Omega^{n-1}\Sigma^n X.
 \]
\end{theorem}

The stable splitting theorem (\fullref{Snaith}) generalizes.

\begin{theorem}
 Under the same condition as above, we have the following natural stable
 homotopy equivalences
 \begin{equation}
  E_n(CX,X) \shot \bigvee_{j=1}^{\infty}
   \widetilde{\mathcal{E}}_n(j;CX,X)/{\Sigma_j}.
   \label{SSforE_n}
 \end{equation}
\end{theorem}

With these theorems, we can replace the pullback diagram
\eqref{PathLoop} by the diagram
\[
\begin{diagram}
\node{\bigvee_{j=1}^{\infty} \Cu_{n}(j)\wedge_{\Sigma_j} X^{\wedge j}}
 \arrow[2]{e} \arrow{s} \node{} \node{\bigvee_{j=1}^{\infty}
 \widetilde{\mathcal{E}}_n(j;CX,X)/{\Sigma_j}} \arrow{s} \\
\node{\{*\}} \arrow[2]{e} \node{} \node{\bigvee_{j=1}^{\infty}
 \Cu_{n-1}(j)\wedge_{\Sigma_j} (\Sigma X)^{\wedge j}}
\end{diagram}
\]
which is stably homotopy equivalent to \eqref{PathLoop}.

In \cite{Tamaki94}, the author defined a $\Sigma_j$--equivariant
decreasing filtrations on $\widetilde{\mathcal{E}}_{n}(j;CX,X)$ which is
compatible with the gravity filtration on $\Cu_{n}(j)$ .
Thus we have a filtration on
\[
\bigvee_{j=1}^{\infty} \widetilde{\mathcal{E}}(j;CX,X)/\Sigma_j
\]
which induces the gravity filtration on
\[
\bigvee_{j=1}^{\infty} \Cu_{n}(j)_+\wedge_{\Sigma_j} X^{\wedge j}.
\]
It is essentially proved in \cite{Tamaki94} that these
filtrations satisfy the condition in \fullref{FilteredFibration}, hence, by \fullref{IsotoEMSS}, the
gravity spectral sequence in \cite{Tamaki94} is isomorphic to the
classical Eilenberg--Moore spectral sequence from the $E^2$--term on.

In order to see this is the case, 
we record the basic properties of the filtrations proved in
\cite{Tamaki94}.

\begin{proposition}
 The above filtrations on $\Cu_{n}(j)$
 and $\widetilde{\mathcal{E}}_{n}(j;CX,X)$
 \begin{multline*}
  \phi = F_{-j-1}\Cu_{n}(j) \subset F_{-j}\Cu_{n}(j) \subset \cdots
   \subset F_{-1}\Cu_{n}(j) = F_0\Cu_{n}(j) = \Cu_{n}(j) \\
  \phi = F_{-j-1}\widetilde{\mathcal{E}}_n(j;CX,X) \subset
   F_{-j}\widetilde{\mathcal{E}}_n(j;CX,X) \subset \cdots \\
  \subset F_{-1}\widetilde{\mathcal{E}}_n(j;CX,X) \subset
   F_0\widetilde{\mathcal{E}}_n(j;CX,X)
 \end{multline*}
 satisfy the following properties.
\begin{enumerate}
 \item The inclusions
       \begin{eqnarray*}
        F_{-q-1}\Cu_{n}(j)_+\wedge_{\Sigma_j} X^{\wedge j} & \subset &
     F_{-q}\Cu_{n}(j)_+\wedge_{\Sigma_j} X^{\wedge j} \\
        F_{-q-1}\widetilde{\mathcal{E}}_n(j;CX,X) & \subset &
     F_{-q}\widetilde{\mathcal{E}}_n(j;CX,X)
       \end{eqnarray*}
       are pointed cofibrations for each $q$.
 \item The inclusion
       \[
        \Cu_{n}(j)_+\wedge X^{\wedge j} \subset \widetilde{\mathcal{E}}_n(j;CX,X)
       \]
       is filtration preserving.
 \item Define
       \begin{eqnarray*}
        F_{-q}C_n(X) & = & \bigvee_{j=1}^{\infty}
     F_{-q}\Cu_{n}(j)_+\wedge_{\Sigma_j} X^{\wedge j} \\
        F_{-q}E_n(CX,X) & = & \bigvee_{j=1}^{\infty}
     F_{-q}\widetilde{\mathcal{E}}_n(j;CX,X)/\Sigma_j.
       \end{eqnarray*}
       Then for each $q$ we have a stable homotopy equivalence:
       \[
        F_{-q}E_n(CX,X)/F_{-q-1}E_n(CX,X) \shot
       F_{-q}C_n(X)/F_{-q-1}C_n(X) \wedge \Omega^{n-1}\Sigma^n X
       \]
\end{enumerate}
\end{proposition}

From these facts it is clear that the gravity filtration on
\[
\bigvee_{j=1}^{\infty} \widetilde{\mathcal{E}}_n(j;CX,X)/\Sigma_j
\]
satisfies the condition in \fullref{FilteredFibration}. Thus we
obtain the following remaining part of \fullref{MainTheorem}.

\begin{corollary}
 If $h_*(\Omega^{n-1}\Sigma^n X)$ is $h_*$--flat, the gravity spectral
 sequence in \cite{Tamaki94} is isomorphic to the classical
 Eilenberg--Moore spectral sequence from the $E^2$--term on.
\end{corollary}

\begin{remark}
 The same argument works for spectral sequences constructed in \cite{Tamaki02}.
\end{remark}

\bibliographystyle{gtart}
\bibliography{link}

\end{document}